\documentclass{amsart}
\usepackage{amsmath,amsfonts,amssymb,graphics}
\usepackage{comment}
\usepackage{cite}

\newtheorem{lemma}{Lemma}[section]

\newtheorem{proposition}[lemma]{Proposition}
\newtheorem{thm}[lemma]{Theorem}
\newtheorem{theorem}[lemma]{Theorem}
\newtheorem{corollary}[lemma]{Corollary}
\theoremstyle{defn}
\newtheorem{defn}[lemma]{Definition}
\newtheorem{example}[lemma]{Example}

\theoremstyle{remark}
\newtheorem{rmk}[lemma]{Remark}

\newtheorem{remark}[lemma]{Remark}
\newtheorem{notation}[lemma]{Notation}

\numberwithin{equation}{section}
\numberwithin{table}{section}

\begin{document}
\title{Intermediate Assouad-like dimensions for measures}
\author{Kathryn E. Hare}
\email{kehare@uwaterloo.ca}
\address{Department of Pure Mathematics, University of Waterloo, Waterloo,
Ontario, Canada N2L 3G1}
\thanks{Research of K.E. Hare was supported by NSERC Grant RGPIN-2016-03719}
\author{Kevin G. Hare}
\email{kghare@uwaterloo.ca}
\address{Department of Pure Mathematics, University of Waterloo, Waterloo,
Ontario, Canada N2L 3G1}
\thanks{Research of K.G. Hare was supported by NSERC Grant RGPIN-2019-03930}

\begin{abstract}
The upper and lower Assouad dimensions of a metric space are local variants
of the box dimensions of the space and provide quantitative information
about the `thickest' and `thinnest' parts of the set. Less extreme versions
of these dimensions for sets have been introduced, including the upper and
lower quasi-Assouad dimensions, $\theta $-Assouad spectrum, and $\Phi $%
-dimensions. In this paper, we study the analogue of the upper and lower $%
\Phi $-dimensions for measures. We give general properties of such
dimensions, as well as more specific results for self-similar measures
satisfying various separation properties and discrete measures.
\end{abstract}

\date{\today }
\maketitle


\section{Introduction}

\label{sec:intro}

\subsection{Background}

The upper and lower Assouad dimensions of a metric space are local variants
of the box dimensions of the space and provide quantitative information
about the `thickest' and `thinnest' parts of the set. The analogous upper
and lower Assouad dimensions for measures, denoted $\dim _{A}\mu $ and $\dim
_{L}\mu $ respectively, were introduced by K\"{a}enm\"{a}ki et al in \cite%
{KL, KLV} and by Fraser and Howroyd in \cite{FH}, where they were originally
called upper and lower regularity dimensions respectively. In recent years,
a number of less extreme versions of these dimensions for sets have been
introduced, including the (upper and lower) quasi-Assouad dimensions, \cite%
{CDW, LX}, the $\theta $-Assouad spectrum, \cite{FY}, and the (most general) 
$\Phi $-dimensions, \cite{GHM}. These dimensions can all be different and
provide more refined information about the geometry of the set.

One reason for the interest in the upper Assouad dimension of a measure is
that it is finite if and only if the measure is doubling, meaning there is
some constant $c>0$ such that $\mu (B(z,r))\geq c\mu (B(z,2r))$ for all $%
z\in \mathrm{supp} \mu $. However, as many interesting measures, such as
self-similar measures that do not satisfy the open set condition, often fail
to be doubling, less extreme dimensional notions for measures may also
provide more insightful information. Hence the motivation for studying the
more moderate (upper and lower) quasi-Assouad dimensions of measures, \cite%
{HHT, HT} and the $\theta $-Assouad spectrum for measures, \cite{FK, FT}. In
this paper, we will introduce and study the analogue of the upper and lower $%
\Phi $-dimensions for measures.

To explain how these dimensions are defined, we recall that for the upper
Assouad dimension of the measure $\mu $, we determine the infimal $\alpha $
such that%
\begin{equation*}
\frac{\mu (B(z,R))}{\mu (B(z,r))}\leq C\left( \frac{R}{r}\right) ^{\alpha }
\end{equation*}%
for all $z\in \mathrm{supp} \mu $ and $r<R$. The lower Assouad dimension is
similar, asking for the supremal $\beta $ such that $\mu (B(z,R))/\mu
(B(z,r))\geq C(R/r)^{\beta }$ for all $z\in \mathrm{supp} \mu $ and $r<R$.
As is the case for the quasi-Assouad dimensions and the $\theta $-Assouad
spectrum, the upper and lower $\Phi $-dimensions, denoted $\overline{\dim }%
_{\Phi }\mu $ and $\underline{\dim }_{\Phi }\mu ,$ are computed by further
restricting the choice of $r$, requiring that $r\leq R^{1+\Phi (R)}$. The
(quasi-) Assouad dimensions and $\theta $-Assouad spectrum are all special
cases of $\Phi $-dimensions. For example, the Assouad dimensions are the
special case of $\Phi =0$.

We refer the reader to Definition \ref{PhiDimMeas} for the precise
definitions of all these notions.

\subsection{Overview of the paper}

In Section 2 we establish basic properties of these dimensions. For example,
we show that 
\begin{equation*}
\dim _{L}\mu \leq \underline{\dim }_{\Phi }\mu \leq \inf_{z}\underline{\dim }%
_{\mathrm{loc}}\mu (z)\leq \sup_{z}\overline{\dim }_{\mathrm{loc}}\mu
(z)\leq \overline{\dim }_{\Phi }\mu \leq \dim _{A}\mu
\end{equation*}%
where $\overline{\dim }_{\mathrm{loc}}\mu (z)$ and $\underline{\dim }_{%
\mathrm{loc}}\mu (z)$ are the upper and lower local dimensions of $\mu $ at $%
z\in \mathrm{supp} \mu $. If the function $\Phi (x)\rightarrow 0$ as $%
x\rightarrow 0$, then the $\Phi $-dimensions lie between the quasi-Assouad
and Assouad dimensions of the measure. We show that $\overline{\dim }_{\Phi
}\mu \geq \overline{\dim }_{\Phi }\mathrm{supp} \mu $ and $\underline{\dim }%
_{\Phi }\mu \leq \underline{\dim }_{\Phi }\mathrm{supp} \mu $ if $\mu $ is a
doubling measure. Examples are given to see that all these inequalities can
be strict.

It is clear that if $\Phi (x)\geq \Psi (x)$ for all $x>0$, then $\overline{%
\dim }_{\Phi }\mu \leq \overline{\dim }_{\Psi }\mu ,$ and conversely for the
lower dimensions. In Proposition \ref{comparison} we prove, more
specifically, that if there exists $\lambda <1$ such that $\Phi (x)\geq
\lambda \Psi (x)$ for all $x$, then%
\begin{equation*}
\overline{\dim }_{\Psi }\mu \geq \lambda \overline{\dim }_{\Phi }\mu \text{
and }\underline{\dim }_{\Psi }\mu \leq \underline{\dim }_{\Phi }\mu
+(1-\lambda )\dim _{A}\mu .
\end{equation*}%
It follows that if $\lim_{x\rightarrow 0}\Phi (x)/\Psi (x)\rightarrow 1,$
then the upper $\Phi $ and $\Psi $-dimensions coincide, as do the lower $%
\Phi $ and $\Psi $-dimensions if $\mu $ is doubling. Moreover, if $\Phi $
and $\Psi $ are both constant functions, then $\overline{\dim }_{\Psi }\mu $
and $\overline{\dim }_{\Phi }\mu $ are simultaneously finite for any measure 
$\mu$.

In Theorem \ref{thm:box}, bounds are given for the $\Phi $-dimensions in
terms of the exponents $s,t$ satisfying $cr^{t}\leq \mu (B(z,r))\leq Cr^{s}$
for all $z\in \mathrm{supp} \mu $ and all $r$. Indeed, if $\Phi
(x)\rightarrow \infty $ as $x\rightarrow 0$, then $\overline{\dim }_{\Phi
}\mu $ is the infimum of such $t$ and $\underline{\dim }_{\Phi }\mu $ is the
supremum of such $s$ if, in addition, $\mu $ is doubling. This improves upon
results in \cite{FK}.

In \cite{FT}, it is asked if an absolutely continuous measure with positive
lower Assouad dimension has its density function in $L^{p}$ for some $p>1$.
In Proposition \ref{FTQues}, we answer this in the negative.

In Section \ref{sec:self-similar}, we study the dimensional properties of
self-similar measures. In contrast to the case for general measures, and
even for self-similar measures satisfying the open set condition, in Theorem %
\ref{SSC} we see that if $\mu $ is a self-similar measure satisfying the
strong separation condition, then 
\begin{equation*}
\sup_{z\in \mathrm{supp} \mu }\overline{\dim }_{\mathrm{loc}}\mu (z)=%
\overline{\dim }_{\Phi }\mu \text{ and }\underline{\dim }_{\Phi }\mu
=\inf_{z\in \mathrm{supp} \mu }\underline{\dim }_{\mathrm{loc}}\mu (z).
\end{equation*}%
In Theorem \ref{T:doubling}, we characterize the finiteness of $\overline{%
\dim }_{\Phi }\mu$, in terms of a doubling-like property, for any
self-similar measure $\mu$ on $\mathbb{R}$ whose support is an interval and
which satisfies the weak separation condition. Consequently, any
equicontractive, self-similar measure $\mu $ with interval support and
satisfying the weak separation condition has the property that $\overline{%
\dim }_{\Phi }\mu <\infty $ for all non-zero, constant functions $\Phi $.
If, in addition, the probabilities associated with the left-most and
right-most contractions in the underlying IFS are minimal, then we even have 
$\overline{\dim }_{\Phi }\mu <\infty $ for any $\Phi $ satisfying $\Phi
(x)/\Psi (x)\rightarrow \infty $ as $x\rightarrow 0$ for $\Psi (x)=\log
|\log x|/|\log x|$. In particular, $\dim _{qA}\mu <\infty $ for such
measures $\mu $. We give an example to show that the function $\Psi (x)$ is
sharp. We also prove that $\dim _{qA}\mu =\infty $ for any biased Bernoulli
convolution with contraction factor the inverse of the golden mean, thus the
extra assumption on the probabilities is also a necessary condition.

In \cite{FH}, Fraser and Howroyd compute the upper Assouad dimension of
discrete measures of the form $\mu =\Sigma _{n=1}^{\infty }p_{n}\delta
_{a_{n}},$ for summable $p_{n}$ either of the form $n^{-\lambda }$ for $%
\beta ^{-n},$ and for $a_{n}$ of similar form. Of course, the support of
such a measure is $\{a_{n}\}_{n=1}^{\infty }\cup \{0\}$. These measures are
the focus of Section \ref{sec:discrete}. We extend the results of \cite{FH}
to all the upper $\Phi $-dimensions and allow $\mu \{0\}>0$. The
relationship between $a_{n}$, $p_{n}$ and the $\Phi $-dimension proves to be
somewhat intricate, often depending on the limiting behaviour of $\Phi (x)$
as it relates to $a_{n}$ and $p_{n}$.

\section{Basic Properties of the $\Phi $-Dimensions}

\label{sec:basic}

\subsection{Definitions}

\begin{defn}
A map $\Phi :(0,1)\mathbb{\rightarrow R}^{+}$ is called a \textbf{dimension
function} if $x^{1+\Phi (x)}$ decreases to $0$ as $x$ decreases to $0$. We
will write $\mathcal{D}$ for the space of all dimension functions.
\end{defn}

Special examples of dimension functions include the constant functions $\Phi
(x)=\delta \geq 0$ and the functions $\Phi (x)=1/|\log x|$ or $\left\vert
\log x\right\vert $. It is useful to observe that as $x^{1+\Phi (x)}\leq x$, 
$x^{1+\Phi (x)}\rightarrow 0$ as $x\rightarrow 0$ for any $\Phi \in \mathcal{%
D}$.

\begin{notation}
Given a bounded metric space $X,$ we denote the open ball centred at $x\in X$
with radius $R$ by $B(x,R)$. The notation $N_{r}(X)$ will mean the least
number of balls of radius $r$ that cover $X$. We write $\mathrm{diam} E$ for
the diameter of $E\subseteq X$.
\end{notation}

\begin{notation}
\label{N:comparable}When we write $f\sim g$ we mean there are constants $%
a,b>0$ such that $af(x)\leq g(x)\leq bf(x)$ for all $x$ in the domain of the
functions $f,g$. When we write $f\preceq g$ we mean there is a constant $c$
such that $f(x)\leq cg(x)$ for all $x$.
\end{notation}

\begin{defn}[Garci\'{a}, Hare, Mendivil \protect\cite{GHM}]
\label{DefnPhidim} Let $\Phi $ be a dimension function. The \textbf{upper }%
and \textbf{lower }$\Phi $\textbf{-dimensions\ }of $E\subseteq X$ are given
by 
\begin{equation*}
\overline{\dim }_{\Phi }E=\inf \left\{ 
\begin{array}{c}
\alpha :(\exists C_{1},C_{2}>0)(\forall 0<r\leq R^{1+\Phi (R)}\leq R<C_{1})%
\text{ } \\ 
N_{r}(B(z,R)\bigcap E)\leq C_{2}\left( \frac{R}{r}\right) ^{\alpha }\text{ }%
\forall z\in E%
\end{array}%
\right\}
\end{equation*}%
and 
\begin{equation*}
\underline{\dim }_{\Phi }E=\sup \left\{ 
\begin{array}{c}
\alpha :(\exists C_{1},C_{2}>0)(\forall 0<r\leq R^{1+\Phi (R)}\leq R<C_{1})%
\text{ } \\ 
N_{r}(B(z,R)\bigcap E)\geq C_{2}\left( \frac{R}{r}\right) ^{\alpha }\text{ }%
\forall z\in E%
\end{array}%
\right\} .
\end{equation*}
\end{defn}

\begin{remark}
{\ }

\begin{enumerate}
\item The \textbf{upper Assouad} and \textbf{lower Assouad dimensions }of $E$%
, \cite{A2, L1}, and denoted $\dim _{A}E$ and $\dim _{L}E$, are the special
cases of the upper and lower $\Phi $-dimensions with $\Phi =0$.

\item If we let $\Phi _{\theta }(x)=$ $1/\theta -1$ for all $x\,,$ then the
upper and lower $\Phi _{\theta }$-dimensions are (slight modifications) of
the \textbf{upper} and \textbf{lower $\theta $-Assouad spectrum} introduced
in \cite{FY}.

\item The \textbf{upper} and \textbf{lower quasi-Assouad dimensions},
denoted $\dim _{qA}E$ and $\dim _{qL}E$ and introduced in \cite{CDW, LX},
can be defined as the limit as $\theta \rightarrow 1$ of the upper and lower 
$\Phi _{\theta }$-dimensions, respectively.
\end{enumerate}
\end{remark}

A metric space $X$ has finite upper Assouad dimension if and only if it is
doubling, meaning there is a constant $M$ such that any ball in $X$ of
radius $R$ can be covered by at most $M$ balls of radius $R/2$, \cite{He}.
The space $X$ has positive lower Assouad dimension if and only if it is
uniformly perfect, meaning there is a constant $c>0$ so that $%
B(z,r)\diagdown B(z,cr)\neq \emptyset $ whenever $z\in X$ and $R$ is at most
the diameter of $X$, \cite{KLV}.

By a measure we will always mean a Borel probability measure on the metric
space $X$. The analogues of the Assouad dimensions for measures (also known
as the upper and lower regularity dimensions), the quasi-Assouad dimensions
and the $\theta $-Assouad spectrum for measures have been extensively
studied, c.f., \cite{FH, FK, FT, HHT, HT, KL, KLV}. Motivated by these
notions, we introduce the larger class of upper and lower $\Phi $-dimensions
for measures.

\begin{defn}
\label{PhiDimMeas} Let $\Phi $ be a dimension function and let $\mu $ be a
measure on the metric space $X$.

The \textbf{upper }and \textbf{lower }$\Phi $\textbf{-dimensions\ }of $\mu $
are given by 
\begin{equation*}
\overline{\dim }_{\Phi }\mu =\inf \left\{ 
\begin{array}{c}
\alpha :\text{ }(\exists \text{ }C_{1},C_{2}>0)(\forall 0<r<R^{1+\Phi
(R)}\leq R\leq C_{1})\text{ } \\ 
\frac{\mu (B(x,R))}{\mu (B(x,r))}\leq C_{2}\left( \frac{R}{r}\right)
^{\alpha }\text{ }\forall x\in \mathrm{supp} \mu%
\end{array}%
\right\}
\end{equation*}%
and 
\begin{equation*}
\underline{\dim }_{\Phi }\mu =\sup \left\{ 
\begin{array}{c}
\alpha :\text{ }(\exists \text{ }C_{1},C_{2}>0)(\forall 0<r<R^{1+\Phi
(R)}\leq R\leq C_{1}) \\ 
\frac{\mu (B(x,R))}{\mu (B(x,r))}\geq C_{2}\left( \frac{R}{r}\right)
^{\alpha }\text{ }\forall x\in \mathrm{supp} \mu%
\end{array}%
\right\} .
\end{equation*}
\end{defn}

\begin{remark}
{\ }

\begin{enumerate}
\item The \textbf{upper} and \textbf{lower Assouad dimensions} of $\mu $,
introduced by K\"{a}enm\"{a}ki et al in \cite{KL, KLV} and Fraser and
Howroyd in \cite{FH}, and denoted $\dim _{A}\mu $ and $\dim _{L}\mu $
respectively, are the upper and lower $\Phi $-dimensions with $\Phi $ the
constant function $0$.

\item If we let $\Phi _{\theta }=1/\theta -1,$ then $\overline{\dim }_{\Phi
_{\theta }}\mu $ and \underline{$\dim $}$_{\Phi _{\theta }}\mu $ are the 
\textbf{upper} and \textbf{lower $\theta$-Assouad spectrum}. The \textbf{%
upper} and \textbf{lower quasi-Assouad dimensions} of $\mu $ are given by 
\begin{equation*}
\dim _{qA}\mu =\lim_{\theta \rightarrow 1}\overline{\dim }_{\Phi _{\theta
}}\mu \text{, }\dim _{qL}\mu =\lim_{\theta \rightarrow 1}\underline{\dim }%
_{\Phi _{\theta }}\mu .
\end{equation*}%
See \cite{HHT, HT}.
\end{enumerate}
\end{remark}

To be precise, the $\theta $-Assouad spectrum of a set $E$, as introduced in 
\cite{FY}, only required consideration of $r=R^{1/\theta }$. However, it was
shown in \cite{FHHTY} that if we denote this dimension by $\overline{\dim }%
^{=\theta }E,$ then $\overline{\dim }_{\Phi _{\theta }}E=\sup_{\psi \leq
\theta }\overline{\dim }^{=\psi }E$. The analogous statements were proved
for the lower $\theta $-Assouad spectrum of sets in \cite{CWC} and for the
upper and lower $\theta $-Assouad spectrum of measures in \cite{HT}.

We note that the same proof as given for sets in \cite[Prop. 2.15]{GHM}
shows that given any measure $\mu ,$ there are dimension functions $\Phi
,\Psi $ such that $\overline{\dim }_{\Phi }\mu =\dim _{qA}\mu $ and $%
\underline{\dim }_{\Psi }\mu =\dim _{qL}\mu $.

\begin{remark}
\label{up}
{\ }
\begin{enumerate}
\item \label{up 1} It is known (see \cite{FH}) that a measure has finite
upper Assouad dimension if and only if it is \textbf{doubling,} meaning
there is a constant $C$ such that 
\begin{equation*}
\mu (B(z,2R))\leq C\mu (B(z,R))\text{ for all }R\leq \mathrm{diam} X,\text{ }%
z\in \mathrm{supp} \mu .
\end{equation*}

\item \label{up 2} If a measure has an atom, then all of its lower $\Phi $%
-dimensions are $0$. More generally, a measure $\mu $ has positive lower
Assouad dimension if and only if $\mu $ is\textbf{\ uniformly perfect} (see 
\cite{HT, KL}) which means there are positive constants $a,b$ such that $\mu
(B(z,R)\diagdown B(z,aR))\geq b\mu (B(z,R))$ for all $z\in \mathrm{supp} \mu 
$ and $R\leq \mathrm{diam} X $ or, equivalently, there are constants $c>1$
and $a>0$ such that 
\begin{equation*}
\mu (B(z,R))\geq c\mu (B(z,aR)\text{ for all }R\leq \mathrm{diam} X,\text{ }%
z\in \mathrm{supp} \mu .
\end{equation*}
\end{enumerate}
\end{remark}

\subsection{Preliminary results}

Here are some easy facts about these dimensions.

\begin{proposition}
{\ } \label{P:basic}

\begin{enumerate}
\item \label{P:basic 1} For all dimension functions $\Phi ,$%
\begin{equation}
\dim _{L}\mu \leq \underline{\dim }_{\Phi }\mu \leq \overline{\dim }_{\Phi
}\mu \leq \dim _{A}\mu  \label{R1}
\end{equation}%
and 
\begin{equation}
\overline{\dim }_{\Phi }\mu \geq \overline{\dim }_{\Phi }\ supp \mu .
\label{R2}
\end{equation}%
If $\mu $ is doubling, then 
\begin{equation}
\underline{\dim }_{\Phi }\mu \leq \underline{\dim }_{\Phi }\mathrm{supp} \mu
.  \label{R3}
\end{equation}

\item \label{P:basic 2} If $\Phi (x)\leq \Psi (x)$ for all $x>0$, then 
\begin{equation}
\overline{\dim }_{\Psi }\mu \text{ }\leq \overline{\dim }_{\Phi }\mu \text{
and }\underline{\dim }_{\Phi }\mu \leq \underline{\dim }_{\Psi }\mu .
\label{R4}
\end{equation}

\item \label{P:basic 3} If $\Phi (x)\rightarrow 0$ as $x\rightarrow 0,$ then 
\begin{equation}
\underline{\dim }_{\Phi }\mu \leq \dim _{qL}\mu \text{ and }\dim _{qA}\mu
\leq \overline{\dim }_{\Phi }\mu \text{ . }  \label{R5}
\end{equation}

\item \label{P:basic 4} If there exists $x_{0}>0$ such that $\Phi (x)\leq
C/\left\vert \log x\right\vert $ for $0<x\leq x_{0},$ then $\overline{\dim }%
_{\Phi }\mu =\dim _{A}\mu $ and $\underline{\dim }_{\Phi }\mu =\dim _{L}\mu $%
.
\end{enumerate}
\end{proposition}

\begin{proof}
The first statement in \ref{P:basic 1} is obvious. For the second, we remark
that it was shown \cite{FH, HHT} that $\overline{\dim }_{\Phi }\mathrm{supp}
\mu $ is dominated by both $\dim _{A}\mu $ and $\dim _{qA}\mu $. The same
arguments work here for the upper $\Phi $-dimensions. For the lower $\Phi $%
-dimensions the arguments are similar to those found in \cite{HT, KL} for
the special cases of the (quasi-) lower Assouad dimensions.

To prove the claims of \ref{P:basic 4}, it is enough to study $\mu
(B(z,R))/\mu (B(z,r))$ for $R\geq r\geq R^{1+\Phi (R)}\geq e^{-C}R$. For
such $r,$ we have%
\begin{equation*}
1\leq \frac{\mu (B(z,R))}{\mu (B(z,r))}\leq \frac{\mu (B(z,R))}{\mu
(B(z,R^{1+\Phi (R)}))}.
\end{equation*}%
Since $R/r\sim 1,$ it follows that the Assouad and $\Phi $-dimensions
coincide.

Statements \ref{P:basic 2} and \ref{P:basic 3} follow easily from the
definitions.
\end{proof}

\begin{remark}
The inequalities of \eqref{R1}-\eqref{R5} can all be strict. In \cite{HHT},
examples are given to show that $\overline{\dim }_{(q)A}\mu >\overline{\dim }%
_{(q)A}\mathrm{supp} \mu $. Similar examples can be constructed for the
upper and lower $\Phi $-dimensions to see the strictness in \eqref{R2} and %
\eqref{R3}. In \cite[Theorem 3.5]{GHM}, formulas are given for the upper and
lower $\Phi $-dimensions of central Cantor sets. Using these formulas many
examples are given there to illustrate the strictness of the analogues of
the inequalities in \eqref{R1}, \eqref{R4} and \eqref{R5} when the measure $%
\mu $ is replaced by the Cantor set $E$. However, if $\mu $ is the uniform
Cantor measure on the Cantor set $E,$ then the upper or lower $\Phi $%
-dimension of $\mu $ coincides with that of $E$.

In fact, it is shown in \cite{GHM} that given any $0<\alpha <\beta <1,$
there is a central Cantor set $E\subseteq \lbrack 0,1]$ with 
\begin{equation*}
\{\overline{\dim }_{\Phi }E:\Phi \in \mathcal{D}\text{, }\lim_{x\rightarrow
0}\Phi (x)=0\}=[\alpha ,\beta ]=[\dim _{qA}E,\dim _{A}E].
\end{equation*}%
Taking $\mu $ to be the uniform Cantor measure on this Cantor set gives the
same interval for the set of upper $\Phi $-dimensions of $\mu $ for $\Phi
\rightarrow 0$.

Many other examples illustrating the strictness of these inequalities are
given throughout this paper.
\end{remark}

We recall the definition of the local dimension of a measure.

\begin{defn}
The \textbf{lower local dimension} of a measure $\mu $ at a point $z$ in its
support is defined as 
\begin{equation*}
\underline{\dim }_{\mathrm{loc}}\mu (z)=\liminf_{r\rightarrow 0}\frac{\log
\mu (B(z,r))}{\log r}.
\end{equation*}%
By replacing $\liminf $ with $\limsup $ we obtain the \textbf{upper local
dimension} and if the lower and upper local dimensions are equal, then we
call the quantity the \textbf{local dimension} of $\mu $ at $z$.
\end{defn}

\begin{proposition}
For any dimension function $\Phi $, \label{basic}%
\begin{equation}
\underline{\dim }_{\Phi }\mu \leq \inf_{z\in \mathrm{supp} \mu }\underline{%
\dim }_{\mathrm{loc}}\mu (z)\leq \dim _{H}\mu \leq \sup_{z\in \mathrm{supp}
\mu }\overline{\dim }_{\mathrm{loc}}\mu (z)\leq \overline{\dim }_{\Phi }\mu .
\label{locdimorder}
\end{equation}
\end{proposition}

\begin{proof}
The middle inequalities around $\dim _{H}\mu $ are standard, see \cite[ch. 10%
]{Fa}. The inequality $\sup_{z}\overline{\dim }_{\mathrm{loc}}\mu (z)\leq 
\overline{\dim }_{\Phi }\mu $ obviously holds if $d=\overline{\dim }_{\Phi
}\mu =\infty $, so assume $d < \infty$. Fix $\varepsilon >0$ and choose $%
C_{1}$ such that for all $r\leq R^{1+\Phi (R)}$ $\leq $ $R\leq C_{1}$ we
have 
\begin{equation*}
\frac{\mu (B(z,R))}{\mu (B(z,r))}\leq \left( \frac{R}{r}\right)
^{d+\varepsilon }.
\end{equation*}%
Taking logarithms and dividing by $-\log r$, we see that 
\begin{equation*}
\frac{\log (\mu (B(z,R))}{-\log r}+\frac{\log (\mu (B(z,r))}{\log r}\leq
(d+\varepsilon )\frac{\log R}{-\log r}+d+\varepsilon .
\end{equation*}%
Keeping $R$ fixed and letting $r\rightarrow 0$ gives%
\begin{equation*}
\overline{\lim }_{r\rightarrow 0}\frac{\log (\mu (B(z,r)))}{\log r}\leq
d+\varepsilon .
\end{equation*}%
That proves $\sup_{z}\overline{\dim }_{\mathrm{loc}}\mu (z)\leq d$. The
argument for the lower $\Phi $-dimension is similar.
\end{proof}

Here is an example illustrating strictness in (\ref{locdimorder}).

\begin{example}[ Examples where $\overline{\dim }_{\Phi }\protect\mu %
>\sup_{z} \overline{\dim }_{\mathrm{loc}}\protect\mu (z)$ or $\protect%
\underline{\dim } _{\Phi }\protect\mu <\inf_{z}\protect\underline{\dim }_{%
\mathrm{loc}}\protect\mu (z)$]
\label{ex:loc vs Phi} We will construct probability measures with support in 
$[0,1]$ by specifying the measure of each of the triadic subintervals of $%
[0,1]$.

Consider a level $n,$ triadic subinterval, $[a/3^{n},(a+1)/3^{n}]$  for
integers $a \in \{ 0,1, \dots, 3^n-1\}$. This can be decomposed into 3
subintervals of level $n+1$, namely $[3a/3^{n+1},(3a+1)/3^{n+1}]$, $%
[(3a+1)/3^{n+1},(3a+2)/3^{n+1}]$ and $[(3a+2)/3^{n+1},3(a+1)/3^{n+1}]$. We
will call these the left child, the middle child and the right child,
respectively, of the original parent interval. We will define the measures
by proscribing the ratio of the measure of a child with the measure of the
parent.

We begin by choosing an increasing sequence $\{n_{j}\}$, with $n_{j+1}\gg
n_{j}$. 
Let $M_{n_{j}}:=[3^{-n_{j}},2\cdot 3^{-n_{j}}]$. Inductively, let $%
M_{n_{j}+k+1}$ be the middle child of $M_{n_{j}+k}$ for $k=0,1,\dots ,n_{j}$
and let $L_{n_{j}+k+1}$ (resp. $R_{n_{j}+k+1}$) be the left (resp. right)
child of $M_{n_{j}+k}$.

Given a sequence $\{p_{j}\}_{j=1}^{\infty }$ with $0\leq p_{j}\leq 1$, we
define the measure $\mu _{\{p_{j}\}}=\mu $ by the rule that the ratio of the 
$\mu $-measure of the middle child $M_{n_{j}+k+1}$ to the measure of its
parent is $p_{j}$ for $k=0,\dots ,n_{j}$ and the ratio of the $\mu $-measure
of the left or right child to the measure of its parent is $\frac{1-p_{j}}{2}
$. Thus, for $k=0,1,...,n_{j}$ 
\begin{equation*}
\mu (M_{n_{j}+k+1})=p_{j}\mu (M_{n_{j}+k})\text{ }
\end{equation*}%
and 
\begin{equation*}
\mu (L_{n_{j}+k+1})=\mu (R_{n_{j}+k+1})=\left( \frac{1-p_{j}}{2}\right) \mu
(M_{n_{j}+k}).
\end{equation*}%
For all other children of all other parents we set the ratio of the measure
of the child to the parent to be $1/3$.

One can see that for any nested sequence of triadic intervals the ratio of
the measure of a parent to a child is 1/3, except possibly a finite number
of times. This gives us that 
\begin{equation*}
\dim _{\mathrm{loc}}\mu (z)=1\text{ for all }z\in \lbrack 0,1].
\end{equation*}

Put $\Phi (x)=1$. Let $x_{j} = \frac{1}{2} 3^{-n_j +1}$ be the midpoint of
the triadic subinterval $M_{n_{j}}=[3^{-n_{j}},2\cdot 3^{-n_{j}}]$. Put $%
R_{j}=3^{-n_{j}}/2$ and $r_{j}=3^{-2n_{j}}/6$. We note that $\Phi (R_{j})=1$%
, thus $R_{j}^{1+\Phi (R_{j})}=3^{-2n_{j}}/4=3r_{j}/2$. As $B(x_{j},r_{j})$
is a triadic interval at level $2n_{j}+1$, and it and all its ancestors back
to level $n_{j}+1$ are middle children, we see that 
\begin{equation*}
\frac{\mu (B(x_{j},R_{j}))}{\mu (B(x_{j},r_{j}))}=\left( \frac{1}{p_{j}}%
\right) ^{n_{j}+1}=\left( 3^{n_{j}+1}\right) ^{-\log _{3}p_{j}}.
\end{equation*}%
Since $R_{j}/r_{j}=3^{n_{j}+1}$, it follows that%
\begin{equation*}
\overline{\dim }_{\Phi }\mu \geq \max (1,-\log _{3}(\liminf_{j}p_{j})).
\end{equation*}%
In a similar fashion, we have that 
\begin{equation*}
\underline{\dim }_{\Phi }\mu \leq \min (1,-\log _{3}(\limsup_{j}p_{j})).
\end{equation*}%
In fact, we have equality in both cases, as similar reasoning shows. We
leave these details to the reader.

Here are some explicit examples.

\begin{itemize}
\item If $\liminf_{j}p_{j}=1/4$ then 
\begin{equation*}
1=\sup_{z\in \mathrm{supp} \mu }\overline{\dim }_{\mathrm{loc}}\mu (z)<%
\overline{\dim }_{\Phi }\mu =\log 4/\log 3.
\end{equation*}

\item If $\liminf_{j}p_{j}=0$ then 
\begin{equation*}
1=\sup_{z\in \mathrm{supp} \mu }\overline{\dim }_{\mathrm{loc}}\mu (z)<%
\overline{\dim }_{\Phi }\mu =\infty .
\end{equation*}

\item If $\limsup_{j}p_{j}=1/2$ then 
\begin{equation*}
\log 2/\log 3=\underline{\dim }_{\Phi }\mu <\inf_{z\in \mathrm{supp} \mu }%
\underline{\dim }_{\mathrm{loc}}\mu (z)=1
\end{equation*}

\item If $\limsup_{j}p_{j}=1$ then 
\begin{equation*}
0=\underline{\dim }_{\Phi }\mu <\inf_{z\in \mathrm{supp}\mu }\underline{\dim 
}_{\mathrm{loc}}\mu (z)=1
\end{equation*}
\end{itemize}
\end{example}

\subsection{Comparing $\Phi $-Dimensions}

As commented earlier, it is immediate from the definition that if $\Phi \geq
\Psi ,$ then $\overline{\dim }_{\Phi }\mu \leq \overline{\dim }_{\Psi }\mu $
and conversely for the lower dimensions. If we know more about the relative
sizes of $\Phi $ and $\Psi ,$ more can be said about the corresponding
dimensions.

\begin{proposition}
\label{comparison}Let $\Phi ,\Psi \in \mathcal{D}$. Suppose there are
constants $0<\lambda <1$ and $x_{0}>0$ such that 
\begin{equation*}
\Phi (x)\geq \lambda \Psi (x)\text{ for all }0<x\leq x_{0}\text{.}
\end{equation*}%
Then for any measure $\mu $ on $E$ we have

\begin{enumerate}
\item \label{comparision 1} $\overline{\dim }_{\Psi }\mu \geq \lambda 
\overline{\dim }_{\Phi }\mu $ and $\underline{\dim }_{\Psi }\mu \leq 
\underline{\dim }_{\Phi }\mu +(1-\lambda )\dim _{A}\mu ;$

\item \label{comparison 2} $\overline{\dim }_{\Psi }E\geq \lambda \overline{%
\dim }_{\Phi }E$ and $\underline{\dim }_{\Psi }E\leq \underline{\dim }_{\Phi
}E+(1-\lambda )\dim _{A}E.$
\end{enumerate}
\end{proposition}

\begin{proof}
\ref{comparision 1} We begin with the upper dimensions. We can assume $%
\overline{\dim }_{\Phi }\mu >0$ for otherwise there is nothing to prove.
Choose positive real numbers $\alpha _{n}\nearrow $ $\overline{\dim }_{\Phi
}\mu $, $x_{n}\in \mathrm{supp} \mu ,R_{n}\rightarrow 0$ and $r_{n}\leq
R_{n}^{1+\Phi (R_{n})}$ such that 
\begin{equation*}
\frac{\mu (B(x_{n},R_{n}))}{\mu (B(x_{n},r_{n}))}\geq \left( \frac{R_{n}}{%
r_{n}}\right) ^{\alpha _{n}}.
\end{equation*}

If there is a subsequence $(n_{j})$ such that $r_{n_{j}}\leq
R_{n_{j}}^{1+\Psi (R_{n_{j}})},$ then it is clear that $\overline{\dim }%
_{\Psi }\mu \geq \sup_{j}\alpha _{n_{j}}=\overline{\dim }_{\Phi }\mu $.
Otherwise, for all but finitely many $n$ we have%
\begin{equation*}
R_{n}^{1+\Psi (R_{n})}\leq r_{n}\leq R_{n}^{1+\Phi (R_{n})}.
\end{equation*}%
Hence%
\begin{eqnarray*}
\frac{\mu (B(x_{n},R_{n}))}{\mu (B(x_{n},R_{n}^{1+\Psi (R_{n})}))} &\geq &%
\frac{\mu (B(x_{n},R_{n}))}{\mu (B(x_{n},r_{n}))}\geq \left( \frac{R_{n}}{%
r_{n}}\right) ^{\alpha _{n}}\geq R_{n}^{-\Phi (R_{n})\alpha _{n}} \\
&=&R_{n}^{-\Psi (R_{n})\left( \frac{\Phi (R_{n})}{\Psi (R_{n})}\alpha
_{n}\right) }\geq R_{n}^{-\Psi (R_{n})\lambda \alpha _{n}}
\end{eqnarray*}%
and this implies that $\overline{\dim }_{\Psi }\mu \geq \lambda \overline{%
\dim }_{\Phi }\mu $.

Now we consider the lower dimensions. We can assume $\underline{\dim }_{\Phi
}\mu <\infty $ and $\dim _{A}\mu <\infty $. Suppose $x_{n}\in \mathrm{supp}
\mu ,R_{n}\rightarrow 0$ and $r_{n}\leq R_{n}^{1+\Phi (R_{n})}$ are chosen
such that 
\begin{equation*}
\frac{\mu (B(x_{n},R_{n}))}{\mu (B(x_{n},r_{n}))}\leq \left( \frac{R_{n}}{%
r_{n}}\right) ^{\alpha _{n}}
\end{equation*}%
where $\alpha _{n}\searrow \underline{\dim }_{\Phi }\mu ,$ and again assume
that for all but finitely many $n$ we have%
\begin{equation*}
R_{n}^{1+\Psi (R_{n})}\leq r_{n}\leq R_{n}^{1+\Phi (R_{n})}.
\end{equation*}%
Then, for any $\varepsilon >0$ and small enough $R_{n},$ we have%
\begin{eqnarray*}
\frac{\mu (B(x_{n},R_{n}))}{\mu (B(x_{n},R_{n}^{1+\Psi (R_{n})}))} &=&\frac{%
\mu (B(x_{n},R_{n}))}{\mu (B(x_{n},R_{n}^{1+\Phi (R_{n})}))}\cdot \frac{\mu
(B(x_{n},R_{n}^{1+\Phi (R_{n})}))}{\mu (B(x_{n},R_{n}^{1+\Psi (R_{n})}))} \\
&\leq &C\frac{\mu (B(x_{n},R_{n}))}{\mu (B(x_{n},r_{n}))}R_{n}^{(\Phi
(R_{n})-\Psi (R_{n}))(\dim _{A}\mu +\varepsilon )} \\
&\leq &C\left( \frac{R_{n}}{r_{n}}\right) ^{\alpha _{n}}R_{n}^{-(1-\lambda
)\Psi (R_{n})(\dim _{A}\mu +\varepsilon )} \\
&\leq &CR_{n}^{-\Psi (R_{n})(\alpha _{n}+(1-\lambda )(\dim _{A}\mu
+\varepsilon ))},
\end{eqnarray*}%
and this obviously implies $\underline{\dim }_{\Psi }\mu \leq \underline{%
\dim }_{\Phi }\mu +(1-\lambda )\dim _{A}\mu $.

\ref{comparison 2} The proof for sets is essentially the same.
\end{proof}

We have the following corollaries as an immediate consequence.

\begin{corollary}
{\ } \label{Cor1}

\begin{enumerate}
\item \label{Cor1 1} If $\Phi (x)/\Psi (x)\rightarrow 1$ as $x\rightarrow 0$%
, then $\overline{\dim }_{\Phi }\mu =\overline{\dim }_{\Psi }\mu $. The same
statement holds for the lower dimensions if, in addition, $\mu $ is doubling.

\item \label{Cor1 2} If $\Phi (x)\rightarrow \theta \neq 0$ as $x\rightarrow
0,$ then $\overline{\dim }_{\Phi }\mu =\overline{\dim }_{\Phi _{\theta }}\mu 
$ where $\Phi _{\theta }$ is the constant function $\theta $.

\item \label{Cor1 3} If $\Phi \sim \Psi ,$ then $\overline{\dim }_{\Phi }\mu
<\infty $ if and only if $\overline{\dim }_{\Psi }\mu <\infty $. In
particular, if $\Phi $ and $\Psi $ are positive constant functions, then $%
\overline{\dim }_{\Phi }\mu <\infty $ if and only if $\overline{\dim }_{\Psi
}\mu <\infty $.
\end{enumerate}
\end{corollary}

\begin{remark}
It would be interesting to know if the assumption of a doubling measure is
necessary for the second statement of \ref{Cor1 1}.
\end{remark}

If $\Phi (x)/\Psi (x)$ does not tend to $1,$ we do not, in general, have
equality of the dimensions as the next result illustrates.

\begin{proposition}
Suppose $\Phi ,\Psi $ are dimension functions decreasing to $0$ as $%
x\rightarrow 0$ with $\Psi (x)|\log x|\rightarrow \infty $ as $x\rightarrow
0 $. Assume there is some constant $\eta >0$ such that $\Phi (x)\geq (1+\eta
)\Psi (x)$ for all $x$ small. Then there is a measure $\mu $ such that $%
\overline{\dim }_{\Phi }\mu <$ $\overline{\dim }_{\Psi }\mu $.
\end{proposition}

\begin{proof}
This is essentially a consequence of \cite[Theorem 3.8]{GHM} where the
analogous result was shown for sets. Indeed, it is shown that under these
assumptions, there is a central Cantor set $E$ with $\overline{\dim }_{\Phi
}E<$ $\overline{\dim }_{\Psi }E$. If we choose $\mu $ to be the uniform
measure on this Cantor set, then $\overline{\dim }_{\Phi }\mu =\overline{%
\dim }_{\Phi }E,$ while $\overline{\dim }_{\Psi }\mu =$ $\overline{\dim }%
_{\Psi }E$.
\end{proof}

\begin{rmk}
We recall that the condition $\limsup_{x\rightarrow 0}\Psi (x)\left\vert
\log x\right\vert <\infty $ implies that the $\Psi $-dimension coincides
with the Assouad dimension (Proposition \ref{P:basic}\ref{P:basic 4}), hence
the necessity of the hypothesis $\Psi (x)|\log x|\rightarrow \infty $. Later
in the paper (Cor. \ref{abequality}\ref{abequality 2}), we will prove that
if $\Phi ,\Psi \rightarrow \infty ,$ then $\overline{\dim }_{\Phi }\mu =%
\overline{\dim }_{\Psi }\mu $ for all measures $\mu ,$ regardless of the
comparative sizes of $\Phi ,\Psi $.
\end{rmk}

What might be thought of as the analogue of Cor. \ref{Cor1}\ref{Cor1 3} for
the lower dimension (namely, that $\underline{\dim }_{\Phi }\mu >0$ if and
only if $\underline{\dim }_{\Psi }\mu >0$ when $\Phi \sim \Psi $) need not
be true, even when the Assouad dimension of the measure is finite, as the
next example illustrates.

\begin{example}[An example of a doubling measure and dimension functions $%
\Phi \sim \Psi ,$ with $\protect\underline{\dim }_{\Phi }\protect\mu =0,$
but $\protect\underline{\dim }_{\Psi }\protect\mu >0$]
We will let $\Phi ,\Psi $ be the constant functions $1,2$ respectively.
Choose a sequence of integers $\{n_{j}\}$ with $n_{j+1}\geq 9n_{j}$ and take
as $\mu $ the corresponding measure given in Example \ref{ex:loc vs Phi}
with $p_{j}=1$ for all $j$. Thus $\underline{\dim }_{\Phi }\mu =0$.

To see that $\mu $ is a doubling measure, consider any $x\in \mathrm{supp}%
\mu $ and the balls $B(x,3^{-n})$ and $B(x,3^{-(n+1)})$. The smaller of
these balls contains a triadic interval $I$ at level $n+1$ which contains $x$%
. (If there is a choice for $I,$ choose one of positive $\mu $-measuure.)
The parent, $J,$ of $I$ is a triadic interval of level $n$ and has the
property that $\mu (J)/3\leq \mu (I)\leq \mu (J)$. The adjacent triadic
intervals at level $n$, say $J^{-}$ and $J^{+},$ have measure either equal
to that of $\mu (J)$ or $0$.

We have that $I\subseteq B(x,3^{-n-1})\subseteq B(x,3^{-n})\subseteq
J^{-}\cup J\cup J^{+}$. This gives that 
\begin{equation*}
\frac{\mu (B(x,3^{-n}))}{\mu (B(x,3^{-(n+1)}))}\leq \frac{\mu (J^{-}\cup
J\cup J^{+})}{\mu (I)}\leq 9.
\end{equation*}%
Thus $\mu $ is doubling and hence has finite Assouad dimension.

We will now show that $\underline{\dim }_{\Psi }\mu \geq 1/4>0$. Let $x\in 
\mathrm{supp}\mu $ and $r\leq R^{1+\Psi (R)}=R^{3}$. Pick $n$ maximal and $N$
minimal such that 
\begin{equation*}
B(x,r)\subset I_{n}\subset I_{N}\subset B(x,R),
\end{equation*}%
where $I_{n}$ and $I_{N}$ are triadic intervals of length $3^{-n}$ and $%
3^{-N}$ respectively. Choose a sequence of triadic intervals $I_{k},$ of
level $k,$ containing $x,$ so that $I_{n}\leq I_{k+1}\subseteq
I_{k}\subseteq I_{N}$ for each $k=N,...,n-1$.

We remark that as $n\sim 3N$ and $n_{j+1}\geq 9n_{j},$ there can be at most
one choice of $j$ with $\{n_{j},...,2n_{j}\}\bigcap \{N,...,n\}$ non-empty.
By the construction of $\mu ,$ for $k\in \{N,...,n\},$ either $\mu (I_{k})=$ 
$\mu (I_{k-1})/3$ (the measure of the child is $1/3$rd that of the parent)
or $\mu (I_{k})=\mu (I_{k-1})$ (the measure of the child equals that of the
parent), with equality only on levels $k$ where $n_{j}\leq k\leq 2n_{j}$ for
this (unique) choice of $n_{j}$. Hence, for all $N\leq k<n_{j}$ and all $%
2n_{j}<k\leq n$ we have $\mu (I_{k})=\mu (I_{k-1})/3$. One can check that at
least $1/4$ of these children will have measure equal to $1/3$ the measure
of their parent and this gives that 
\begin{equation*}
\frac{\mu (B(x,R))}{\mu (B(x,r))}\geq \frac{\mu (I_{N})}{\mu (I_{n})}\geq
\left( \frac{1}{3}\right) ^{(N-n)/4}.
\end{equation*}%
As $R/r\sim 3^{n-N},$ it follows that $\underline{\dim }_{\Phi }\mu \geq
1/4. $
\end{example}

We remark that it is possible to have $\overline{\dim }_{\Phi }\mu <\infty $
for all non-zero, constant dimension functions and yet $\dim _{qA}\mu
=\infty $. In Proposition \ref{BCbiased} we will prove that the biased
Bernoulli convolution with contraction factor the inverse of the golden mean
has this property, as does the measure in the next example.

\begin{example}[A measure $\protect\mu $ having $\dim _{qA}\protect\mu %
=\infty ,$ but $\overline{\dim }_{\Phi }\protect\mu <\infty $ for all $\Phi =%
\protect\theta >0$]
\label{FTNotqA}We define the measure $\mu $ on the diadic subintervals of $%
[0,1]$ by specifying that the ratio of the measure of the left child of a
diadic subinterval to that of its parent is $2/3,$ while the ratio of the
measure of the right child to the parent is $1/3$.

Let $r=2^{-(n+[\theta n]+2)}$, $R=2^{-n}+2^{-(n+[\theta n]+2)}$ and take $x $
to be the midpoint of the diadic interval of level $n+[\theta n]+1$ that has 
$1/2$ as its right end point. Then $B(x,r)$ is this diadic interval, while $%
B(x,R)$ contains the diadic interval of level $n$ to the right of $1/2 $ and
is contained in the union of this level $n$ diadic interval and the level $%
n-1$ diadic interval immediately to the left of $1/2$. Thus $R/r\sim
2^{\theta n}$ and 
\begin{equation*}
\frac{\mu (B(x,R))}{\mu (B(x,r))}\sim \frac{(2/3)^{n}}{(1/3)^{n+\theta n}}.
\end{equation*}%
Hence%
\begin{equation*}
\overline{\dim }_{\theta }\mu \sim \frac{\log 2+\theta \log 3}{\theta \log 2}%
=\frac{1}{\theta }+\frac{\log 3}{\log 2}.
\end{equation*}%
This tends to infinity as $\theta \rightarrow 0$, thus $\dim _{qA}\mu
=\infty $.
\end{example}

\subsection{Regularity-like properties}

Recall that a measure is called $s$-regular if there is a constant $c>0$
such that for all $x\in \mathrm{supp} \mu $ and $r\leq \mathrm{diam}(\mathrm{%
supp} \mu)$ we have $c^{-1}r^{s}\leq \mu (B(x,r))\leq cr^{s}$. Clearly, all
the $\Phi $-dimensions agree for regular measures, c.f., \cite{KL, KLV}, but
the converse is not true, as seen in \cite[Example 2.7]{HT}.

Following \cite{FK}, we will define the \textbf{upper Minkowski dimension}
of a compactly supported measure $\mu $ to be%
\begin{equation*}
\overline{\dim }_{M}\mu =\inf \{t:\exists B>0\text{ so that}\inf_{z\in 
\mathrm{supp}\mu }\mu (B(z,r))\geq Br^{t}\text{ }\forall r\leq \mathrm{diam}(%
\mathrm{supp} \mu)\}
\end{equation*}%
and the \textbf{Frostman dimension} of $\mu $ to be 
\begin{equation*}
\dim _{F}\mu =\sup \{s:\exists A>0\text{ so that}\sup_{z\in \mathrm{supp}\mu
}\mu (B(z,r))\leq A r^{s}\text{ }\forall r\leq \mathrm{diam}(\mathrm{supp}
\mu)\}.
\end{equation*}%
Note that $\sup_{z}\{\overline{\dim }_{\mathrm{loc}}\mu (z)\}\leq \overline{%
\dim }_{M}\mu $ and $\inf_{z}\{\underline{\dim }_{\mathrm{loc}}\mu (z)\}\geq
\dim _{F}\mu $.

In \cite{FK}, Fraser and K\"{a}enm\"{a}ki show that for the constant
function $\Phi =1/\theta -1$, $\underline{\dim }_{\Phi }\mu \leq \dim
_{F}\mu $ and $\overline{\dim }_{M}\mu \leq $ $\overline{\dim }_{\Phi }\mu
\leq (\overline{\dim }_{M}\mu) /(1-\theta )$. Here is an extension of this
result.

\begin{thm}
\label{thm:box} Let $\mu $ be a measure with compact support and suppose $%
\Phi \in \mathcal{D}.$ Put $L=\limsup_{x\rightarrow 0}\Phi (x)^{-1}$ . Then 
\begin{equation*}
\dim _{F}\mu \geq \underline{\dim }_{\Phi }\mu \geq \dim _{F}\mu -L(%
\overline{\dim }_{M}\mu -\dim _{F}\mu )
\end{equation*}%
and%
\begin{equation*}
\overline{\dim }_{M}\mu \leq \overline{\dim }_{\Phi }\mu \leq \overline{\dim 
}_{M}\mu +L(\overline{\dim }_{M}\mu -\dim _{F}\mu ).
\end{equation*}
\end{thm}

\begin{proof}[Proof of Theorem \protect\ref{thm:box}]
We first observe that for any fixed $\rho >0,$%
\begin{equation*}
\inf \{\mu (B(z,\rho )):z\in \mathrm{supp}\mu \}>0.
\end{equation*}%
This is an elementary compactness argument. Assume not. Then for some $\rho
>0$ there exists a sequence $z_{n}\in \mathrm{supp}\mu $ such that $%
z_{n}\rightarrow z_{0}\in \mathrm{supp}\mu $ and $\mu (B(z_{n},\rho
))\rightarrow 0$. Choose $N$ such that for all $n\geq N$ we have $\Vert
z_{n}-z_{0}\Vert \leq \rho /2$. Then $B(z_{0},\rho /2)\subset B(z_{n},\rho )$
and hence $\mu (B(z_{0},\rho /2)\leq \mu (B(z_{n},\rho ))$. This implies
that $\mu (B(z_{0},\rho /2))=0$, a contradiction to $z_{0}$ being in the
support of $\mu $.

Let $D=\overline{\dim }_{\Phi }\mu $ and $d=\underline{\dim }_{\Phi }\mu $.
We will first prove the left side inequalities. Of course, the second is
obvious if $D=\infty ,$ so assume otherwise. Fix $0<\varepsilon <1$ and
choose $\rho $ such that for all $r\leq R^{1+\Phi (R)}\leq R\leq \rho $ and $%
z\in \mathrm{supp}\mu ,$ 
\begin{equation*}
C_{1}\left( \frac{R}{r}\right) ^{d-\varepsilon }\leq \frac{\mu (B(z,R))}{\mu
(B(z,r))}\leq C_{2}\left( \frac{R}{r}\right) ^{D+\varepsilon }\text{ .}
\end{equation*}

Assume $r\leq \rho ^{1+\Phi (\rho )}$. For some constant $C_{\rho }>0$ we
have 
\begin{equation*}
\frac{C_{\rho }}{\mu (B(z,r))}\leq \frac{\mu (B(z,\rho ))}{\mu (B(z,r))}\leq
C_{2}\left( \frac{\rho }{r}\right) ^{D+\varepsilon }=C_{2}\rho
^{D+\varepsilon }r^{-(D+\varepsilon )}.
\end{equation*}%
Consequently, for a suitable constant $B,$ 
\begin{equation*}
\mu (B(z,r))\geq C_{\rho }C_{2}^{-1}\rho ^{-(D+\varepsilon
)}r^{D+\varepsilon }=Br^{D+\varepsilon }.
\end{equation*}%
As this is true for all $\varepsilon >0,$ we deduce that $\overline{\dim }%
_{M}\mu \leq D$.

We can similarly conclude that $\dim _{F}\mu \geq d$ since 
\begin{equation*}
\frac{1}{\mu (B(z,r))}\geq \frac{\mu (B(z,\rho ))}{\mu (B(z,r))}\geq
C_{1}\left( \frac{\rho }{r}\right) ^{d-\varepsilon }=C_{1}\rho
^{d-\varepsilon }r^{-(d-\varepsilon )}.
\end{equation*}

Now we prove the right side inequalities. For notational ease, put $a=\dim
_{F}\mu $ and $b=$ $\overline{\dim }_{M}\mu $. There is no loss of
generality in assuming $b<\infty $. Take $\varepsilon >0$. For any $q$ and $%
z\in \mathrm{supp}\mu $ we have $\mu (B(z,q))\leq Aq^{a-\varepsilon }$ and $%
\mu (B(z,q))\geq Bq^{b+\varepsilon }$ for positive constants $A,B$ depending
on $\varepsilon $.

Suppose $r=R^{1+\Psi (R)}$ with $\Psi (R)\geq \Phi (R)$. Then for $C=B/A$ we
have 
\begin{eqnarray*}
\frac{\mu (B(z,R))}{\mu (B(z,r))} &\geq &\frac{B}{A}\left( \frac{%
R^{b+\varepsilon }}{r^{a-\varepsilon }}\right) =CR^{b+\varepsilon
-(a-\varepsilon )(1+\Psi (R))} \\
&=&CR^{-\Psi (R)(a-\varepsilon -(b-a+2\varepsilon )/\Psi (R))}=C\left( \frac{%
R}{r}\right) ^{a-\varepsilon -(b-a+2\varepsilon )/\Psi (R)} \\
&\geq &C\left( \frac{R}{r}\right) ^{a-\varepsilon -(b-a+2\varepsilon )/\Phi
(R)}.
\end{eqnarray*}%
As this holds for all $\varepsilon >0,$ it follows that 
\begin{equation*}
d\geq \liminf_{R\rightarrow 0}(a-(b-a)/\Phi (R)))=a-L(b-a).
\end{equation*}

The argument for $D$ is similar.
\end{proof}

The following corollaries are immediate.

\begin{corollary}
{\ } \label{abequality}

\begin{enumerate}
\item \label{abequality 1}  If $\Phi (x)\rightarrow \infty $ as $%
x\rightarrow 0$, then $\overline{\dim }_{\Phi }\mu =$ $\overline{\dim }%
_{M}\mu $. If, in addition, $\overline{\dim }_{M}\mu <\infty ,$ then $%
\underline{\dim }_{\Phi }\mu =\dim _{F}\mu $.

\item \label{abequality 2} If $\Phi _{1},\Phi _{2}\rightarrow \infty $, then 
$\overline{\dim }_{\Phi _{1}}\mu =$ $\overline{\dim }_{\Phi _{2}}\mu $. If $%
\overline{\dim }_{M}\mu <\infty ,$ then also $\underline{\dim }_{\Phi
_{1}}\mu =\underline{\dim }_{\Phi _{2}}\mu $.
\end{enumerate}
\end{corollary}

\begin{remark}
{\ }

\begin{enumerate}
\item It was shown in \cite[Prop. 2.8]{GHM} that if $\Phi (x)\rightarrow
\infty $ as $x\rightarrow 0$, then $\overline{\dim }_{\Phi }E$ is the upper
box (or Minkowski) dimension of $E,$ while $\underline{\dim }_{\Phi }E$ is
the lower box dimension if, in addition, $\underline{\dim }_{\Phi }E>0$.

\item We do not know if the assumption that $\overline{\dim }_{M}\mu <\infty 
$ is necessary.
\end{enumerate}
\end{remark}

\begin{corollary}
\label{thm:loc vs box vs theta} Let $\Psi _{\theta }=1/\theta -1$ for $%
\theta \in (0,1)$. Then 
\begin{equation*}
\overline{\dim }_{M}\mu \leq \overline{\dim }_{\Psi _{\theta }}\mu \leq 
\frac{\overline{\dim }_{M}\mu -\theta \dim _{F}\mu }{1-\theta }
\end{equation*}%
and 
\begin{equation*}
\frac{\dim _{F}\mu -\theta \overline{\dim }_{M}\mu }{1-\theta }\leq 
\underline{\dim }_{\Psi _{\theta }}\mu \leq \dim _{F}\mu .
\end{equation*}

Furthermore, $\lim_{\theta \rightarrow 0}\overline{\dim }_{\Psi _{\theta
}}\mu =\overline{\dim }_{M}\mu $ and if $\overline{\dim }_{M}\mu <\infty ,$
then $\lim_{\theta \rightarrow 0}\underline{\dim }_{\Psi _{\theta }}\mu
=\dim _{F}\mu .$
\end{corollary}

Recall that in Proposition \ref{basic} it was shown that $\overline{\dim }%
_{\Phi }\mu \geq \sup_{z\in \mathrm{supp}\mu }\dim_{\mathrm{loc}}\mu (z)$.
Another consequence of the theorem is that we can show it is possible to
have $\overline{\dim }_{\Phi }\mu =\infty $ for all $\Phi \in \mathcal{D}$
and yet $\dim_{\mathrm{loc}}\mu (z)\leq 1$ for all $z\in \mathrm{supp}\mu $.

\begin{example}[A measure $\protect\mu $ with $\dim _{\mathrm{loc}}\protect%
\mu (z)=1$ for all $z,$ but $\overline{\dim }_{M}\protect\mu =\infty $]
We can achieve this with a slight modification of the strategy of Example %
\ref{ex:loc vs Phi}. Instead of assigning special weights $p_{j}$ on levels $%
n_{j}+k$ for $k=0,...,n_{j}$, do this on levels $n_{j}+k$ for $%
k=0,...,n_{j}^{2}$ and choose $n_{j+1}\gg n_{j}^{2}$. Let $\Phi
(x)=\left\vert \log _{3}x\right\vert \rightarrow \infty $ as $x\rightarrow 0$%
. By choosing $p_{j}$ with $\liminf_{j}p_{j}=1$, we can construct $\mu $
with the property that $\dim _{\mathrm{loc}}\mu (z)=1$ for all $z\in \mathrm{%
supp}\mu ,$ but $\overline{\dim }_{\Phi }\mu =\infty $. Since $\Phi
(x)\rightarrow \infty $ as $x\rightarrow 0$, we have $\overline{\dim }%
_{M}\mu =\infty $ from Corollary \ref{abequality} and therefore $\overline{%
\dim }_{\Psi }\mu =\infty $ for all dimension functions $\Psi $.

Similarly, by taking $p_{j}$ with $\limsup_{j}p_{j}=0$, we can construct a
measure $\nu $ with $\underline{\dim }_{\Phi }\nu =\dim _{F}\nu =0$ for all
dimension functions $\Phi ,$ while $\dim _{\mathrm{loc}}\nu (z)=1$ for all $%
z $.
\end{example}

\subsection{ Smoothness Properties}

In \cite{FT}, Fraser and Troscheit show that if $\mu $ is a uniformly
perfect, absolutely continuous measure supported on $[0,1],$ with monotonic
density function $f,$ then $f\in L^{p}(\mathbb{R)}$ for some $p>1$. They
asked if this was true without the monotonicity assumption. Here we will
show that the answer to this question is no.

In this subsection (only), we will think of $[0,1]$ both as a subset of $%
\mathbb{R}$ and as the group $\mathbb{T}$ under addition mod 1. When we
consider balls in the latter sense, we will use the notation $B_{\mathbb{T}}$%
.

When we say a measure on $[0,1]$ is symmetric, we will mean that $\mu
(E)=\mu (1-E)$ for all Borel sets $E\subseteq \lbrack 0,1]\subseteq \mathbb{R%
}$. (Of course, if we view $[0,1]$ as $\mathbb{T}$, then $1-E=-E$.)

\begin{lemma}
\label{L1}Let $\mu $ be a measure supported on $[0,1]\subseteq \mathbb{R}$
that is symmetric.

\begin{enumerate}
\item \label{L1 1} If there are constants $a,c>0$ such that $\mu
(B(z,R))\geq c\mu (B(z,aR)) $ for all $z\in $ $\mathrm{supp}$ $\mu $ and $%
R\leq 1$, then 
\begin{equation*}
\mu (B_{\mathbb{T}}(z,R))\geq \frac{c}{2}\mu (B_{\mathbb{T}}(z,aR))\text{
for all }z\in \text{$\mathrm{supp}$}\mu \text{ and }R\leq 1\text{.}
\end{equation*}

\item \label{L1 2} Similarly, if there are constants $a,c>0$ such that $\mu
(B_{\mathbb{T}}(z,R))\geq c\mu (B_{\mathbb{T}}(z,aR))$ for all $z\in $ $%
\mathrm{supp}$ $\mu $ and $R\leq 1$, then 
\begin{equation*}
\mu (B(z,R))\geq \frac{c}{2}\mu (B(z,aR))\text{ for all }z\in \text{$\mathrm{%
supp}$}\mu \text{ and }R\leq 1\text{.}
\end{equation*}
\end{enumerate}
\end{lemma}

\begin{proof}
First, note that $\mu (B(z,R))=\mu (B(z,R)\bigcap [0,1])\leq \mu (B_{\mathbb{%
T}}(z,R))$ for all $z\in \lbrack 0,1]$ and $R\leq 1$. If $B(z,R)\bigcap
[0,1]=B(z,R),$ then $\mu (B(z,R))=\mu (B_{\mathbb{T}}(z,R))$. Otherwise, if $%
z\leq 1<z+R$, then $B_{\mathbb{T}}(z,R)=(z-R,1]\bigcup [0,z+R-1)$. From the
symmetry of $\mu ,$ it follows that 
\begin{equation*}
\mu (B_{\mathbb{T}}(z,R))=\mu ((z-R,1])+\mu ((1-(z+R-1),1]).
\end{equation*}%
But $1-(z+R-1)\geq z-R,$ hence $\mu (B_{\mathbb{T}}(z,R))\leq 2\mu (B(z,R))$%
. The argument if $z-R<0\leq R$ is similar. Consequently, we also have $\mu
(B(z,R))\geq \frac{1}{2}\mu (B_{\mathbb{T}}(z,R))$.

Both parts follow easily from these observations.
\end{proof}

\begin{lemma}
\label{L2}Let $\mu $ be the uniform Cantor measure on the classical
middle-third Cantor set $C$.

\begin{enumerate}
\item \label{L2 1} For every $z\in C$ and $R\leq 1$ we have 
\begin{equation*}
\mu (B(z,R))\geq 8\mu (B(z,R/3^{4})).
\end{equation*}

\item \label{L2 2} For every $z\in \lbrack 0,1]$ and $R\leq 1$ we have 
\begin{equation*}
\mu (B(z,R))\geq 8\mu (B(z,R/3^{5})).
\end{equation*}
\end{enumerate}
\end{lemma}

\begin{remark}
We emphasize that in \ref{L2 1}, the bound holds for all $z \in C$ while in  %
\ref{L2 2}, it must hold for all $z \in [0,1]$.
\end{remark}

\begin{proof}
\ref{L2 1} First, suppose $R=3^{-N}$ and let $z\in C$. Consider the Cantor
intervals $I_{j}$ of levels $j=N,$ $N+3$ that contain $z$. Note that $%
B(z,3^{-N})$ contains $\mathrm{int} I_{N}$ and since the gaps adjacent to $%
I_{N+3}$ have length at least $3^{-(N+3)}$, $B(z,3^{-(N+3)})\bigcap
C\subseteq I_{N+3} $. Consequently,%
\begin{equation*}
\mu (B(z,3^{-N}))\geq 2^{-N}\text{ and }\mu (B(z,3^{-(N+3)}))\leq 2^{-(N+3)}%
\text{.}
\end{equation*}%
Hence $\mu (B(z,3^{-N}))\geq 8$ $\mu (B(z,3^{-(N+3)}))$.

Now suppose $0<R\leq 1$ and the integer $N$ is chosen with $3^{-(N+1)}<R\leq
3^{-N}$. Suppose $z\in C$. Then%
\begin{equation*}
\mu (B(z,R))\geq \mu (B(z,3^{-(N+1)}))\geq 8\mu (B(z,3^{-(N+4)}))\geq 8\mu
(B(z,R/3^{4}))\text{.}
\end{equation*}

\ref{L2 2} If $z\in C$ there is nothing to prove, so assume otherwise. If $%
B(z,R/3^{5})\bigcap C$ is empty there is, again, nothing to prove. So assume
otherwise. Then $z$ belongs to one of the gaps in the construction of the
Cantor set and the distance to the nearest endpoint of that gap, $w,$ is at
most $R/3^{5}$. Hence $B(z,R/3^{5})\subseteq B(w,2R/3^{5})$ and $%
B(w,2R/3)\subseteq B(z,R)$. As $w\in C,$ we can apply part \ref{L2 1} to
deduce that 
\begin{equation*}
\mu (B(z,R/3^{5}))\leq \mu (B(w,2R/3^{5}))\leq \frac{1}{8}\mu
(B(w,2R/3))\leq \frac{1}{8}\mu (B(z,R)).
\end{equation*}
\end{proof}

Suppose $\mu ,\nu $ are measures supported on $[0,1]$. When we write $\mu
\ast \nu $ we will mean the convolution taken over $\mathbb{T}$. Thus $\mu
\ast \nu $ is another measure supported on $[0,1],$ which we can either
think of as a measure on $\mathbb{T}$ or on $\mathbb{R}$.

\begin{lemma}
\label{L3} Suppose $\mu, \nu$ are measures on $[0,1]$. If there are
constants $a,c>0$ such that $\mu (B_{\mathbb{T}}(z,R))\geq c\mu (B_{\mathbb{T%
}}(z,aR))$ for all $z\in \lbrack 0,1]$ and $R\leq 1,$ then%
\begin{equation*}
\mu \ast \nu (B_{\mathbb{T}}(z,R))\geq c\mu \ast \nu (B_{\mathbb{T}}(z,aR))%
\text{ for all }z\in \lbrack 0,1]\text{ and }R\leq 1.
\end{equation*}
\end{lemma}

\begin{proof}
Let $z\in \lbrack 0,1]$ and $R\leq 1$. With addition being understood mod 1,
we have 
\begin{eqnarray*}
\mu \ast \nu (B_{\mathbb{T}}(z,R)) &=&\int \int 1_{B_{\mathbb{T}%
}(z,R)}(x+y)d\mu (x)d\nu (y) \\
&=&\int \mu (B_{\mathbb{T}}(z-y,R))d\nu (y) \\
&\geq &\int c\mu (B_{\mathbb{T}}(z-y,aR))d\nu (y)=c\mu \ast \nu (B_{\mathbb{T%
}}(z,aR)).
\end{eqnarray*}
\end{proof}

Since the Cantor measure $\mu $ is symmetric, combining these lemmas gives
the following useful fact.

\begin{corollary}
\label{C1}If $\mu $ is the uniform Cantor measure on the classical Cantor
set and $\nu $ is any symmetric measure on $[0,1]$, then%
\begin{equation*}
\mu \ast \nu (B(z,R))\geq 2\mu \ast \nu (B(z,R/3^{5}))\text{ for all }z\in
\lbrack 0,1]\text{ and }R\leq 1.
\end{equation*}
\end{corollary}

\begin{proof}
Lemma \ref{L2} gives that $\mu (B(z,R))\geq 8\mu (B(z,R/3^{5}))$ for all $%
z\in \lbrack 0,1]$ and $R\leq 1$. From Lemma \ref{L1}\ref{L1 1}, $\mu (B_{%
\mathbb{T}}(z,R))\geq 4\mu (B_{\mathbb{T}}(z,R/3^{5}))$ and then Lemma \ref%
{L3} implies%
\begin{equation*}
\mu \ast \nu (B_{\mathbb{T}}(z,R))\geq 4\mu \ast \nu (B_{\mathbb{T}%
}(z,R/3^{5}))\text{ for all }z\in \lbrack 0,1]\text{ and }R\leq 1.
\end{equation*}%
To complete the argument, call upon Lemma \ref{L1}\ref{L1 2}.
\end{proof}

We are now ready to answer the question asked by Fraser and Troscheit \cite%
{FT} in the negative.

\begin{proposition}
\label{FTQues}There is an absolutely continuous measure $\nu $ with density
function $f$ having the property that $\dim _{L}\nu >0,$ but $f\notin L^{p}$
for any $p>1$.
\end{proposition}

\begin{proof}
We will give an explicit example. Let $K_{n}$ denote the $n$'th Fejer kernel
on $\mathbb{T=[}0,1]$, 
\begin{equation*}
K_{n}(x)=\sum_{j=-n}^{n}\left( 1-\frac{\left\vert j\right\vert }{n}\right)
e^{ijx},
\end{equation*}%
and inductively define integers $N_{m}\in \{3^{k}\}_{k=1}^{\infty }$ with $%
N_{1}=3$ and $N_{m+1}>3^{2^{m}}N_{m}$. Put%
\begin{equation*}
g(x)=\sum_{m=1}^{\infty }m^{-2}K_{3^{2^{m}}}(N_{m}x).
\end{equation*}%
Note that $g\in L^{1}$ since $\left\Vert K_{n}\right\Vert _{1}=1,$ and $g$
is symmetric. Let $\mu $ denote the uniform Cantor measure on the classical
Cantor set and let $\nu =g\ast \mu $ (where the convolution is on $\mathbb{T}
$).

Since $g\in L^{1},$ the measure $\nu $ is absolutely continuous (whether
viewed as a measure on $\mathbb{T}$ or $\mathbb{R}$). By Corollary \ref{C1} $%
\nu $ is uniformly perfect and hence has positive lower Assouad dimension,
as explained in Remark \ref{up}\ref{up 2}.

We will check that its density function, $f,$ does not belong to $L^{p}(%
\mathbb{T)}$ for any $p>1$ by verifying that the Fourier transform $(%
\widehat{f}(n))_{n=-\infty }^{\infty }\notin \ell ^{q}$ for any $q<\infty $.
An appeal to the Hausdorff-Young inequality will then imply $f\notin L^{p}(%
\mathbb{T)}$ for any $p>1$. Since a function supported on $[0,1]$ belongs to 
$L^{p}(\mathbb{R)}$ if and only if it belongs to $L^{p}(\mathbb{T)}$, this
will complete the argument.

It is immediate from the definitions that 
\begin{equation*}
\widehat{g}(n)=\frac{1}{m^{2}}\left( 1-\frac{\left\vert n\right\vert }{%
3^{2^{m}}}\right) \text{ if }n\in \{\pm 1,...,\pm 3^{2^{m}}\}\cdot N_{m} \ \ 
\text{for some } m \in \mathbb{N}
\end{equation*}%
and that $\widehat{g}(n)=0$ if $n\notin \bigcup\limits_{m}\{0,\pm 1,...,\pm
3^{2^{m}}\}\cdot N_{m}$. Thus $\widehat{g}(n)\geq 2/(3m^{2})$ on $\{\pm
1,...,\pm 3^{2^{m}-1}\}\cdot N_{m}$.

It is well known that$\left\vert \widehat{\mu }(3^{k})\right\vert
=\left\vert \widehat{\mu }(3)\right\vert \neq 0$ for all $k\geq 1$. Thus 
\begin{equation*}
\left\vert \widehat{g}(n)\widehat{\mu }(n)\right\vert =\left\vert \widehat{f}%
(n)\right\vert \geq \frac{2}{3m^{2}}\left\vert \widehat{\mu }(3)\right\vert
\end{equation*}%
for each $n\in S_{m}=\{\{\pm 1,...,\pm 3^{2^{m}-1}\}\cdot N_{m}\}\bigcap
\{3^{k}\}_{k=1}^{\infty }$. Since $\left\vert S_{m}\right\vert =2(2^{m}-1)$
and the choice of the integers $N_{m}$ ensures that the sets $S_{m}$ are
disjoint, we have%
\begin{equation*}
\sum_{n\in \mathbb{Z}}\left\vert \widehat{f}(n)\right\vert ^{q}\geq
\sum_{m=1}^{\infty }\left( \frac{2}{3m^{2}}\left\vert \widehat{\mu }%
(3)\right\vert \right) ^{q}2(2^{m}-1)=\infty
\end{equation*}%
for each $q<\infty $. Thus $(\widehat{f}(n))_{n=-\infty }^{\infty }\notin
\ell ^{q}$ for any $q<\infty $ and that completes the proof.
\end{proof}

\section{$\Phi $-Dimensions of Self-similar measures}

\label{sec:self-similar}

\subsection{Self-similar measures and separation properties\label{S:self-sim}%
}

In this section, our focus will be on self-similar measures that satisfy
various separation conditions. We begin with useful notation.

Consider the iterated function system (IFS), where the maps $%
S_{j}:X\rightarrow X$ are similarities with contraction factors $r_{j}$ for $%
j=0,...,m$ and $m\geq 1$. Assume, also, that we are given probabilities $%
\{p_{j}\}_{j=0}^{m}$, meaning $p_{j}>0$ and $\sum_{j=0}^{m}p_{j}=1$. By the 
\textbf{self-similar measure} $\mu $ associated with the IFS $%
\{S_{j}\}_{j=0}^{m}$ and the probabilities $\{p_{j}\}_{j=0}^{m}$, we mean
the unique probability measure $\mu $ on $X$ satisfying the property that
for any Borel set $E\subseteq X$ we have%
\begin{equation*}
\mu (E)=\sum_{j=0}^{m}p_{j}\mu (S_{j}^{-1}(E))\text{.}
\end{equation*}%
This measure will have as its support $K$, the unique, non-empty, compact
set $K$ satisfying $K=\bigcup\limits_{j=0}^{m}S_{j}(K)$, known as the 
\textbf{self-similar set} associated with the IFS.

Let $\Sigma $ be the set of all finite words on the alphabet $\{0,1,...,m\}$%
. Given $w\in \Sigma ,$ say $w=(j_{1},...,j_{n})$, let $%
w^{-}=(j_{1},...,j_{n-1})$, $S_{w}=S_{j_{1}}\circ \cdot \cdot \cdot \circ
S_{j_{n}},$ 
\begin{equation*}
r_{w}=\prod_{i=1}^{n}r_{j_{i}}\text{ and }p_{w}=\prod_{i=1}^{n}p_{j_{i}}.
\end{equation*}%
Note that $r_{w}$ is the contraction factor of $S_{w}$. Let 
\begin{equation*}
r_{\min}=\min \left\vert r_{j}\right\vert > 0
\end{equation*}%
and put 
\begin{equation*}
\Lambda _{n}=\{w\in \Sigma :\left\vert r_{w}\right\vert \leq r_{\min}^{n}%
\text{ and }\left\vert r_{w^{-}}\right\vert >r_{\min}^{n}\}.
\end{equation*}%
If the IFS consists of equicontractive similarities (all $r_{j}=r_{\min}\in
(0,1)$), then $\Lambda _{n}$ consists of the words $w$ of length $n$. More
generally, there exist $a,b>0$ such that $w\in \Lambda _{n}$ implies $an\leq
\left\vert w\right\vert \leq bn$. Note that for each $n$,%
\begin{equation*}
K=\bigcup\limits_{\sigma \in \Lambda _{n}}S_{\sigma }(K).
\end{equation*}

IFS satisfying the following definitions have been much studied.

\begin{defn}
The IFS $\{S_{j}\}_{j=0}^{m}$, and any associated self-similar measure, are
said to satisfy:

\begin{enumerate}
\item The\textbf{\ strong separation condition} (SSC) if the sets $S_{j}(K)$
are disjoint for $j=0,...,m$;

\item The \textbf{open set condition} (OSC) if there is a bounded,
non-empty, open set $U$ such that $S_{j}(U)\subseteq U$ for each $j$ and the
sets $S_{j}(U)$ are disjoint;

\item The \textbf{weak separation condition} (WSC) if there is some $x_{0}$ $%
\in \mathbb{R}$ and integer $M$ such that for any $n\in \mathbb{N}$ and
finite word $\tau ,$ any closed ball of radius $r_{\min}^{n}$ contains no
more than $M$ distinct points of the form $S_{\sigma }(S_{\tau }(x_{0}))$
for $\sigma \in \Lambda _{n}$.
\end{enumerate}
\end{defn}

The definition we have given of the WSC is a restricted case of the original
definition due to Lau and Ngai, \cite{LN}. Many equivalent properties can be
found in \cite{Ze}.

It is well known that 
\begin{equation*}
SSC\subseteq OSC\subseteq WSC
\end{equation*}%
and that both these inclusions are proper. For example, the IFS with the two
similarities $S_{0}(x)=x/2,$ $S_{1}(x)=x/2+1/2$ on $\mathbb{R}$ satisfies
the OSC, but not the SSC. The IFS $\mathcal{S}_{\rho }=\{S_{0}(x)=\rho
x,S_{1}(x)=\rho x+1-\rho \}$ where $\rho $ is the inverse of a Pisot number
and the IFS $\mathcal{S}_{d}=\{S_{j}(x)=x/d+(d-1)jx/(dm):j=0,1,...,m\}$
where $2\leq d\leq m$ are integers, satisfy the WSC but not the OSC. In the
case of the IFS $\mathcal{S}_{\rho }$, any associated self-similar measure
is known as a Bernoulli convolution and is said to be biased if $p_{0}\neq
p_{1}$. In the case of the IFS $\mathcal{S}_{d},$ for a suitable choice of
probabilities, the self-similar measure is the $m$-fold convolution of the
uniform Cantor measure on the Cantor set with contraction factor $1/d$.

\subsection{Self-similar measures satisfying the strong separation condition}

It was shown in \cite{Fr} that self-similar sets arising from an IFS that
satisfies the open set condition have equal upper and lower Assouad
dimensions (and hence also all $\Phi $-dimensions). This is not true for
self-similar measures. For instance, the measure of Example \ref{FTNotqA} is
the self-similar measure arising from the IFS with $S_{0}(x)=x/2$, $%
S_{1}(x)=x/2+1/2$ and probabilities $2/3,$ $1/3$. This IFS satisfies the
open set condition and yet we have $\dim _{qA}\mu =\infty ,$ while $%
\overline{\dim }_{\Phi }\mu <\infty $ for all non-zero constant functions $%
\Phi $.

However, we cannot produce such an example with a self-similar measure that
satisfies the strong separation property, as our next result shows.

\begin{theorem}
\label{SSC}Assume $\mu $ is a self-similar measure that satisfies the strong
separation condition. For any dimension function $\Phi $ we have%
\begin{equation*}
\underline{\dim }_{\Phi }\mu =\min \{\dim _{\mathrm{loc}}\mu (z):z\in 
\mathrm{supp}\mu \}
\end{equation*}%
and 
\begin{equation*}
\overline{\dim }_{\Phi }\mu =\max \{\dim _{\mathrm{loc}}\mu (z):z\in \mathrm{%
supp}\mu \}.
\end{equation*}
\end{theorem}

\begin{proof}
Assume the measure $\mu $ arises from the IFS $\{S_{j}\}_{j=0}^{m}$ that
satsifies the SSC, with probabilities $\{p_{j}\},$ and that $K$ is the
associated self-similar set. It is well known (see \cite[ch. 11]{Fa}) that
if the contraction factor of $S_{j}$ is $r_{j}$, then 
\begin{equation*}
\{\dim _{\mathrm{loc}}\mu (z):z\in \mathrm{supp}\mu \}=\left[ \min_{j}\frac{%
\log p_{j}}{\log r_{j}},\max_{j}\frac{\log p_{j}}{\log r_{j}}\right]
:=[\theta ,\Theta ].
\end{equation*}%
Of course, this means $r_{j}^{\Theta }\leq p_{j}\leq r_{j}^{\theta }$ for
all $j$.

As the upper and lower $\Phi $-dimensions are bounded (below and above,
respectively) by the maximum and minimum local dimensions (Proposition \ref%
{basic}), it will be enough to show that there are constants $C_{0},C_{1}>0$
such that for all $x\in \mathrm{supp}\mu $, $R\leq \mathrm{diam}(\mathrm{supp%
} \mu)$ and $0<r\leq R^{1+\Phi (R)}$, we have 
\begin{equation*}
C_{0}\left( \frac{R}{r}\right) ^{\theta }\leq \frac{\mu (B(x,R))}{\mu
(B(x,r))}\leq C_{1}\left( \frac{R}{r}\right) ^{\Theta }
\end{equation*}%
to see that $\overline{\dim }_{\Phi }\mu =\Theta $ and $\underline{\dim }%
_{\Phi }\mu =\theta $.

Fix such $x,R$ and $r\leq R^{1+\Phi (R)}$ and choose integers $n,m$ so that $%
r_{\min}^{n}\leq R\leq r_{\min}^{n-1}$ and $r_{\min}^{m}\leq r\leq
r_{\min}^{m-1}$. Obtain $w\in \Lambda _{n}$ and $w\sigma \in \Lambda _{m}$
such that $x\in S_{w\sigma }(K)$. Then%
\begin{equation*}
\left\vert r_{w}\right\vert \leq r_{\min}^{n}\leq R<r_{\min}^{-2}\left\vert
r_{w}\right\vert
\end{equation*}%
and 
\begin{equation*}
\left\vert r_{w\sigma }\right\vert \leq r_{\min}^{m}\leq
r<r_{\min}^{-2}\left\vert r_{w\sigma }\right\vert ,
\end{equation*}%
so 
\begin{equation*}
\frac{R}{r}\geq \frac{\left\vert r_{w}\right\vert }{r_{\min}^{-2}\left\vert
r_{w\sigma }\right\vert }=\frac{r_{\min}^{2}}{\left\vert r_{\sigma
}\right\vert }.
\end{equation*}%
Since $S_{w}(K)\subseteq B(x,R)$ and $S_{w\sigma }(K)\subseteq B(x,r)$ we
have 
\begin{equation*}
\mu (B(x,r))\geq p_{w\sigma }\text{ and }\mu (B(x,R))\geq p_{w}.
\end{equation*}

Because the IFS\ satisfies the strong separation condition, there is some $%
\varepsilon >0$ such that $d(S_{i}(K),S_{j}(K))\geq \varepsilon $ for all $%
i\neq j$. Consequently, for any word $\tau $ and $i\neq j$, $d(S_{\tau
i}(K),S_{\tau j}(K))\geq \varepsilon \left\vert r_{\tau }\right\vert $.

Choose an integer $L$ such that $\varepsilon r_{\min}^{-(L-1)}>2$. Let $W$
be the set of words $v\in \Lambda _{n}$ such that $S_{v}(K)\bigcap
B(x,R)\neq \emptyset $, so $B(x,R)\bigcap K\subseteq \bigcup\limits_{v\in
W}S_{v}(K)$. We claim that the words $v\in W$ must have a common ancestor $%
\tau \in \Lambda _{n-L}$. If not, there would be a pair \thinspace $%
v,v^{\prime }\in W $ with different ancestors at level $n-L$. But, then, 
\begin{equation*}
d(S_{v}(K),S_{v^{\prime }}(K))\geq \varepsilon r_{\min}^{n-L},
\end{equation*}%
which exceeds the diameter of $B(x,R),$ and this is impossible. Thus $%
B(x,R)\bigcap K\subseteq S_{\tau }(K)$ where $\tau $ is the common ancestor.
Moreover, as $w\in W$, $p_{\tau }\leq p_{w}(\min p_{i})^{-L}$, so 
\begin{equation*}
\mu (B(x,R))\leq \mu (S_{\tau }(K))=p_{\tau }\leq p_{w}c_{1}
\end{equation*}%
for $c_{1}=(\min p_{i})^{-L}$.

These facts, together with the definition of $\Theta ,$ implies%
\begin{equation*}
\frac{\mu (B(x,R))}{\mu (B(x,r))}\leq \frac{p_{w}c_{1}}{p_{w\sigma }}\leq 
\frac{c_{1}}{p_{\sigma }}\leq c_{1}\left( \frac{1}{\left\vert r_{\sigma
}\right\vert }\right) ^{\Theta }\leq C_{1}\left( \frac{R}{r}\right) ^{\Theta
}
\end{equation*}%
for a suitable choice of $C_{1}$.

As a similar upper bound can be found for $\mu (B(x,r))$, the lower bound
follows in the same manner.
\end{proof}

\subsection{Self-similar measures satisfying the weak separation condition}

In this subsection we will assume the measure $\mu $ arises from an IFS $%
\{S_{j}\}_{j=0}^{m}$ of similarities $S_{j}(x)=r_{j}x+d_{j}$ on $\mathbb{R}$
that satisfies the WSC. We will also assume that the self-similar set (and
support of the measure) $K=[0,1]$. We continue to use the notation of the
previous subsection.

It was proven in \cite{HHR} that such measures have the property that there
is some $a>0$ such that 
\begin{equation}
\left\vert S_{\sigma }(w)-S_{\tau }(z)\right\vert \geq ar_{\min}^{n}
\label{gap}
\end{equation}%
whenever $\sigma, \tau \in \Lambda _{n}$, $w,z\in \{0,1\}$ and  $S_\sigma(w)
\neq S_\tau(z)$. This property is very helpful in studying the dimensional
properties of $\mu $.

It is convenient to introduce further notation. For each $n\in \mathbb{N}$,
let $h_{1},...,h_{s_{n}}$ denote the set of elements of $\{S_{\sigma }(0),$ $%
S_{\sigma }(1):\sigma \in \Lambda _{n}\},$ listed in increasing order. The
intervals, $[h_{j},h_{j+1}],$ are called the \textbf{net intervals of level }%
$n$\textbf{.} In what follows $\Delta _{n}$ will always denote a net
interval of level $n$ and $\Delta _{n}(x)$ will be a level $n$ net interval
containing $x$ (noting that there could be two choices if $x$ is a boundary
point $h_{i}$.) We write $\ell (I)$ for the length of the interval $I$. From
(\ref{gap}) it follows that 
\begin{equation*}
ar_{\min}^{n}\leq \ell (\Delta _{n})\leq r_{\min}^{n}.
\end{equation*}

Put%
\begin{equation*}
P_{n}(\Delta _{n})=\sum_{\substack{ w\in \Lambda _{n}  \\ %
S_{w}[0,1]\supseteq \Delta _{n}}}p_{w}.
\end{equation*}%
Let%
\begin{equation*}
p=\min p_{j}^{M}
\end{equation*}
where $M$ is the maximum length of any word $w$ such that there exists an
integer $m$ and word $\sigma \in \Lambda _{m-1}$ with $\sigma w\in \Lambda
_{m}$.

The definitions ensure that if $\Delta _{n}\subseteq \Delta _{n-1},$ then 
\begin{equation}
P_{n-1}(\Delta _{n-1})\geq P_{n}(\Delta _{n})\geq pP_{n-1}(\Delta _{n-1}).
\label{Pn}
\end{equation}%
Furthermore, as $\ell (S_{\sigma }[0,1])\leq r_{\min}^{n}$ whenever $\sigma
\in \Lambda _{n},$ we have%
\begin{equation}
\mu (B(x,r_{\min}^{n}))\geq P_{n}(\Delta _{n}(x))\geq \mu (\Delta _{n}(x)).
\label{P/mu}
\end{equation}

It was shown in \cite[Cor. 4.6]{HHT} that these measures $\mu $ satisfy $%
\dim _{qA}\mu <\infty $ if and only if $\mu $ has the doubling-like property
that for every $\varepsilon >0$ there is a constant $C$ such that 
\begin{equation*}
\mu (B(x,R))\leq CR^{-2\varepsilon }\mu (B(x,R/2))
\end{equation*}%
for all $x\in \mathrm{supp}\mu $ and $0<R\leq 1$. Motivated by this, we
introduce the following definition of $\Phi $-doubling.

Recall that a function $\Phi $ is said to be \textbf{doubling} if there is a
constant $c>0$ such that 
\begin{equation*}
\Phi (x)\leq c\Phi (x/2)
\end{equation*}%
whenever $x>0$. Doubling dimension functions include $\Phi =\delta $, $\Phi
(x)=1/|\log x|$ and $\Phi (x)=\log |\log x|/|\log x|$.

\begin{defn}
We will say the measure $\mu $ on $X$ is $\Phi $\textbf{-doubling} if there
are constants $C\geq 1$, $\gamma >0$ such that%
\begin{equation*}
\mu (B(x,R))\leq CR^{-\gamma \Phi (R)}\mu (B(x,R/2))
\end{equation*}%
for all $x\in \mathrm{supp}\mu $ and $0<R\leq 1$.
\end{defn}

Notice that if $\Phi =0,$ this is the usual definition of a doubling measure.

Given $n\in \mathbb{N}$, let%
\begin{equation*}
\phi (n)=n\Phi (r_{\min}^{n})\geq 0.
\end{equation*}%
It is easy to check that if $\Phi $ is a doubling function, then $\mu $ is $%
\Phi $-doubling if and only if there is a (possibly different) constant $%
C\geq 1$ such that 
\begin{equation}
\mu (B(x,r_{\min}^{n}))\leq C^{1+\phi (n)}\mu (B(x,r_{\min}^{n+1}))
\label{D1}
\end{equation}%
for all $x\in \mathrm{supp}\mu $ and $n\in \mathbb{N}$. Note that a repeated
application of (\ref{D1}) shows that for each positive integer $k$, there is
a constant $C_{k}\geq 1$ such that 
\begin{equation*}
\mu (B(x,r_{\min}^{n}))\leq C_{k}^{1+\phi (n)}\mu (B(x,r_{\min}^{n+k})).
\end{equation*}

The property of being $\Phi $-doubling can be described in terms of the
measure of net intervals.

\begin{lemma}
\label{DoublingRem}Assume $\mu $ is a self-similar measure that satisfies
the WSC and has support $[0,1]$. Then $\mu $ is $\Phi $-doubling if and only
if there is a constant $C_{0}\geq 1$ such that 
\begin{equation}
\mu (\Delta _{n})\geq C_{0}^{-(1+\phi (n))}\mu (\Delta _{n}^{\ast })
\label{PhiDoubling}
\end{equation}%
whenever $\Delta _{n}^{\ast }$ is a level $n$ net interval adjacent to the
level $n$ net interval $\Delta _{n}$.
\end{lemma}

\begin{proof}
Fix $a>0$ so that $\ell (\Delta _{n})\geq ar_{\min}^{n}$ for all net
intervals $\Delta _{n}$ of level $n$ and all $n\in \mathbb{N}$.

Suppose $\mu $ is $\Phi $-doubling. Let $\Delta _{n}$ be any level $n$ net
interval and $\Delta _{n}^{\ast }$ be an adjacent net interval. Let $x$
denote the midpoint of $\Delta _{n}$.

The doubling assumption ensures that for a suitable constant $C\geq 1,$ 
\begin{eqnarray*}
\mu (\Delta _{n}) &=&\mu (B(x_{\Delta _{n}},\ell (\Delta _{n})/2))\geq \mu
(B(x_{\Delta _{n}},ar_{\min}^{n}/2)) \\
&\geq &C^{-(1+\phi (n))}\mu (B(x_{\Delta _{n}},2r_{\min}^{n}))\geq
C^{-(1+\phi (n))}\mu (\Delta _{n}^{\ast }),
\end{eqnarray*}%
where the last inequality holds because $B(x_{\Delta
_{n}},2r_{\min}^{n})\supseteq \Delta _{n}^{\ast }$.

Conversely, assume there exists a constant $C_{0}\geq 1$ such that for all $%
n $, $\mu (\Delta _{n})\geq C_{0}^{-(1+\phi (n))}\mu (\Delta _{n}^{\ast })$.
Fix $x\in \lbrack 0,1]$ and suppose $x\in \Delta _{n}$. (If $x$ is a
boundary point of a net interval, choose either net interval.) Let $\Delta
_{n}^{(1)}$ be the level $n$ net interval immediately to its right, and more
generally, let $\Delta _{n}^{(j)}$ be the net interval of level $n$
immediately to the right of $\Delta _{n}^{(j-1)}$ (should it exist). By
repeated application of (\ref{PhiDoubling}), 
\begin{equation*}
\mu (B(x,r_{\min}^{n}))\geq \mu (\Delta _{n})\geq C_{0}^{-(1+\phi (n))}\mu
(\Delta _{n}^{(1)})\geq C_{0}^{-k(1+\phi (n))}\mu (\Delta _{n}^{^{(k)}}).
\end{equation*}%
Choose $k$ so that $[x,x+2r_{\min}^{n}]\bigcap [0,1]\subseteq
\bigcup\limits_{j=0}^{k}\Delta _{n}^{(j)}$; notice $k\leq 1+2/a$. For the
constant $C_{1}=C_{0}^{k}$, we have 
\begin{equation*}
\mu ([x,x+2r_{\min}^{n}])\leq \mu \left( \bigcup\limits_{j=0}^{k}\Delta
_{n}^{(j)}\right) \leq kC_{1}^{1+\phi (n)}\mu (\Delta _{n})\leq
kC_{1}^{1+\phi (n)}\mu (B(x,r_{\min}^{n})).
\end{equation*}%
We similarly bound $\mu ([x-2r_{\min}^{n},x])$ and hence deduce that 
\begin{equation*}
\mu (B(x,r_{\min}^{n}))\geq \frac{1}{2k}C_{1}^{-(1+\phi (n))}\mu
(B(x,2r_{\min}^{n})).
\end{equation*}%
This suffiices to prove that $\mu $ is $\Phi $-doubling.
\end{proof}

We now characterize $\Phi $-doubling in terms of the upper $\Phi $%
-dimensions.

\begin{proposition}
\label{T:doubling} Assume $\mu $ is a self-similar measure on $\mathbb{R}$
that satisfies the weak separation condition and has support $[0,1]$.
Suppose that $\Phi $ is an increasing, doubling, dimension function. Then $%
\overline{\dim }_{\Phi }\mu <\infty $ if and only if $\mu $ is $\Phi $%
-doubling.
\end{proposition}

\begin{proof}
Fix $a>0$ so that $\ell (\Delta _{n})\geq ar_{\min}^{n}$ for all level $n$
net intervals $\Delta _{n}$.

First, suppose that $d=\overline{\dim }_{\Phi }\mu <\infty $ and fix $%
\varepsilon >0$. By the definition of the upper $\Phi $-dimension, there is
a constant $C=C(\varepsilon )$ such that for any suitably large integer $n$
we have 
\begin{equation}
\frac{\mu (B(x,2r_{\min}^{n}))}{\mu (B(x,ar_{\min}^{n+\phi (n)}/2)}\leq
Cr_{\min}^{-}{}^{\phi (n)(d+\varepsilon )}\text{ for all }x\in \lbrack 0,1]%
\text{.}  \label{PhiDim}
\end{equation}

Consider any level $n$ net interval of level $\Delta _{n},$ with midpoint $x$%
. Then 
\begin{equation*}
B(x,ar_{\min}^{n+\phi (n)}/2) \cap [0,1] \subseteq \Delta _{n},
\end{equation*}%
while $B(x,2r_{\min}^{n})$ contains both $\Delta _{n}$ and the two adjacent
level $n$ net intervals. Let $\Delta _{n}^{\ast }$ denote either adjacent
interval. Then (\ref{PhiDim}) gives 
\begin{eqnarray*}
\mu (\Delta _{n}) &\geq &\mu (B(x,ar_{\min}{}^{n+\phi (n)}/2))\geq
C^{-1}r_{\min}{}^{\phi (n)(d+\varepsilon )}\mu (B(x,2r_{\min}^{n})) \\
&\geq &C^{-1}r_{\min}^{\phi (n)(d+\varepsilon )}\mu (\Delta _{n}^{\ast
})\geq C_{1}^{-(1+\phi (n))(d+\varepsilon )}\mu (\Delta _{n}^{\ast })
\end{eqnarray*}%
for $C_{1}=\max (r_{\min}^{-1},C^{1/(d+\varepsilon )})\geq 1$. By Lemma \ref%
{DoublingRem}, $\mu $ is $\Phi $-doubling.

Conversely, assume $\mu $ is $\Phi $-doubling. Fix $x\in \lbrack 0,1]$ and $%
N\in \mathbb{N}$. Let $\Delta _{N}$ denote the level $N$ net interval
containing $x$ (taking either, if there is a choice) and let $\Delta
_{N}^{R} $ and $\Delta _{N}^{L}$ denote the two adjacent, level $N$ net
intervals to the right and left respectively.

According to the Lemma, the $\Phi $-doubling condition implies 
\begin{equation*}
\mu (\Delta _{N}^{R})\leq C_{0}^{1+\phi (N)}\mu (\Delta _{N})
\end{equation*}%
and similarly for $\mu (\Delta _{N}^{L})$. Since $B(x,r_{\min}^{N}a)\cap
\lbrack 0,1]\subseteq \Delta _{N}\cup \Delta _{N}^{R}\cup \Delta _{N}^{L}$, (%
\ref{P/mu}) implies%
\begin{equation*}
\mu (B(x,r_{\min}^{N}a))\leq \mu (\Delta _{N}\cup \Delta _{N}^{R}\cup \Delta
_{N}^{L})\leq 3C_{0}^{1+\phi (N)}\mu (\Delta _{N})\leq 3C_{0}^{1+\phi
(N)}P_{N}(\Delta _{N}).
\end{equation*}%
Choose any integer 
\begin{equation*}
n\geq N(1+\Phi (r_{\min}^{N+1}a)).
\end{equation*}%
Let $\Delta _{n}\subseteq \Delta _{N}$ be the net interval of level $n$
containing $x$. From (\ref{Pn}) and (\ref{P/mu}) we see that%
\begin{equation*}
\frac{\mu (B(x,r_{\min}^{N}a))}{\mu (B(x,r_{\min}^{n}))}\leq 3C_{0}^{1+\phi
(N)}\frac{P_{N}(\Delta _{N}(x))}{P_{n}(\Delta _{n}(x))}\leq 3C_{0}^{1+\phi
(N)}p^{-(n-N)}.
\end{equation*}%
The doubling assumption of $\Phi $ ensures there is some $\beta >0$
(independent of $N)$ such that 
\begin{equation}
\Phi (r_{\min}^{N+1}a)\geq \beta \Phi (r_{\min}^{N}),  \label{beta}
\end{equation}%
so $n-N=N\Phi (r_{\min}^{N+1}a)\geq \beta \phi (N)$. Taking $s,t\geq 0$ such
that $C_{0}=r_{\min}^{-s}$ and $p=r_{\min}^{t},$ we have%
\begin{eqnarray}
\frac{\mu (B(x,r_{\min}^{N}a))}{\mu (B(x,r_{\min}^{n}))} &\leq
&3r_{\min}^{-s(1+\phi (N))}r_{\min}^{-t(n-N)}\leq
3r_{\min}^{-s}r_{\min}^{-(t+s/\beta )(n-N)}  \label{E1} \\
&\leq &C\left( \frac{r_{\min}^{N}a}{r_{\min}^{n}}\right) ^{\alpha }  \notag
\end{eqnarray}%
for $\alpha \geq t+s/\beta $ and another constant $C\geq 1$. That proves $%
\overline{\dim }_{\Phi }\mu \leq \alpha <\infty $.
\end{proof}

\begin{remark}
It would be interesting to know if this result holds for all measures.
\end{remark}

The IFS $\{\rho x,\rho x+1-\rho \},$ where $\rho $ is the inverse of a Pisot
number (such as the golden mean), and the IFS $\{x/d+(d-1)jx/(dm)$ $%
\}_{j=0}^{m},$ for integers $2\leq d\leq m,$ are examples of IFS that do not
satisfy the OSC, but satisfy a separation property stronger than the WSC
known as finite type. This notion was introduced by Ngai and Wang in \cite%
{NW}. For equicontractive IFS of similarities on $\mathbb{R}$ it can be
defined as follows.

\begin{defn}
Let $\mathcal{S=}\{S_{j}\}_{j=0}^{m}$ be an equicontractive IFS of
similarities on $\mathbb{R}$ with contraction factor $0<r_{\min}<1$. The
IFS, or any associated self-similar measure, is said to be of \textbf{finite
type} if there is a finite set $F\subseteq \mathbb{R}$ such that if $v,w$
are words on $\{0,1,...,m\}$ of length $n$, and $c$ is the diameter of the
self-similar set, then either 
\begin{equation*}
\left\vert S_{v}(0)-S_{w}(0)\right\vert >cr_{\min}^{n}\text{ or }%
r_{\min}^{-n}\left\vert S_{v}(0)-S_{w}(0)\right\vert \in F\text{.}
\end{equation*}
\end{defn}

An IFS that is of finite type satisfies the WSC. Conversely, it is proven in 
\cite{HHR} that any equicontractive, self-similar measure that satisfies the
WSC and has support $[0,1]$ is of finite type. Any equicontractive IFS that
satisfies the OSC with the open set being $(0,1)$ is also of finite type.

It is known that an IFS of finite type has the property that there are only
finitely many values for $\ell (\Delta _{n})r_{\min}^{-n},$ over all level $n
$ net intervals and all $n$.

\begin{corollary}
Suppose $\mu $ is any equicontractive, self-similar, finite type measure
with support $[0,1]$. Then $\overline{\dim }_{\Phi }\mu <\infty $ for any
dimension function $\Phi = \delta > 0 $.
\end{corollary}

\begin{proof}
Choose $a>0$ such that $ar_{\min}^{n}\leq $ $\ell (\Delta _{n})\leq
r_{\min}^{n}$ for all level $n$ net intervals $\Delta _{n}$ and fix an
integer $k$ such that $r_{\min}^{k}\leq a/2$.

Let $\Delta _{n}$ be any level $n$ net interval and suppose $x$ is its
midpoint. Choose a word $\omega $ of length $n+k$ so that $x\in S_{\omega
}[0,1]\subseteq \Delta _{n}$. Thus 
\begin{equation*}
\mu (\Delta _{n})\geq \mu (S_{\omega }[0,1])\geq (\min p_{j})^{n+k}\text{.}
\end{equation*}%
It is known that for any finite type measure there is a constant $A$ such
that $\mu (\Delta _{n})\leq A^{n}$, \cite{HHM}. Since $\phi (n)=n\delta $
when $\Phi =\delta$, it easily follows from this that (\ref{PhiDoubling}) is
satisfied for such $\Phi $. Hence $\mu $ is $\Phi $-doubling and therefore
the upper $\Phi $-dimension is finite for all non-zero constant functions $%
\Phi $.
\end{proof}

The measure $\mu $ studied in Example \ref{FTNotqA} is of finite type and
has support $[0,1]$. As we saw in that example, $\overline{\dim }_{\Phi }\mu
<\infty $ for all $\Phi =\delta \neq 0$, but $\dim _{qA}\mu =\infty $,
showing the sharpness of the corollary. The biased Bernoulli convolutions
discussed next are another class of such examples.

\begin{proposition}
\label{BCbiased}Let $\mu $ be the biased Bernoulli convolution arising from
the IFS $\{\rho x$, $\rho x+1-\rho \}$ with probabilities $p,1-p,$ where $%
p>1/2$ and $\rho $ is the inverse of the golden mean. Then $\mu $ is an
equicontractive, self-similar measure of finite type with support $[0,1],$
but $\dim _{qA}\mu =\infty $.
\end{proposition}

\begin{proof}
It is well known that this IFS is of finite type, c.f. \cite{F05}$.$ As
explained there, the net intervals of level $n$ can all be labelled by $n+1$%
-tuples, $(1,\gamma _{1},...,\gamma _{n}),$ where $\gamma _{i}\in
\{2,...,7\} $ (and the allowed choices for $\gamma _{i+1}$ depend on $\gamma
_{i}$) and 
\begin{equation*}
\rho ^{n+3}\leq \ell (\Delta _{n})\leq \rho ^{n}.
\end{equation*}
Two adjacent net intervals of level four are $\Delta _{0}=(1,3,5,6,3)$ and $%
\Delta _{1}=(1,3,5,7,5)$ which lies immediately to its right. The net
interval 
\begin{equation*}
\Delta _{0}^{(k)}:=(1,3,5,6,3,(5,7)^{k},5)
\end{equation*}%
is the right-most descendent of $\Delta _{0}$ at level $5+2k,$ and adjacent
to it is the left-most descendent of $\Delta _{1}$ at the same level, 
\begin{equation*}
\Delta _{1}^{(k)}:=(1,3,5,7,5,(3,5)^{k},3).
\end{equation*}

From the calculations of \cite[Section 4]{HHN} (in the notation used there $%
c_{1}=3$, $c_{2}=5$ and $\overline{c_{1}}=7$), it follows that $\mu (\Delta
_{0}^{(k)})\sim \left\Vert T_{0}^{k}\right\Vert $ and $\mu (\Delta
_{1}^{(k)})\sim \left\Vert T_{1}^{k}\right\Vert $ where 
\begin{equation*}
T_{0}=%
\begin{bmatrix}
p(1-p) & p(1-p) \\ 
0 & (1-p)^{2}%
\end{bmatrix}%
\text{ and }T_{1}=%
\begin{bmatrix}
p^{2} & 0 \\ 
(1-p)^{2} & p(1-p)%
\end{bmatrix}%
,\text{ }
\end{equation*}%
and the matrix norm $\left\Vert T\right\Vert =\sum_{i,j}\left\vert
T_{ij}\right\vert $ when $T=(T_{ij})$.

An induction argument shows that 
\begin{equation*}
(T_{0})^{2^{k}}=%
\begin{bmatrix}
(p(1-p))^{2^{k}} & A_{k} \\ 
0 & (1-p)^{2^{k+1}}%
\end{bmatrix}%
,\text{ }(T_{1})^{2^{k}}=%
\begin{bmatrix}
p^{2^{k+1}} & 0 \\ 
B_{k} & (p(1-p))^{2^{k}}%
\end{bmatrix}%
\end{equation*}%
with 
\begin{eqnarray*}
A_{k} &=&p(1-p)\prod_{i=0}^{k-1}((p(1-p))^{2^{i}}+(1-p)^{2^{i+1}}) \\
&=&p(1-p)\prod_{i=0}^{k-1}(p(1-p))^{2^{i}}\prod_{i=0}^{k-1}\left( 1+\left( 
\frac{1-p}{p}\right) ^{2^{i}}\right) \\
&=&(p(1-p))^{2^{k}}\prod_{i=0}^{k-1}\left( 1+\left( \frac{1-p}{p}\right)
^{2^{i}}\right)
\end{eqnarray*}%
and 
\begin{eqnarray*}
B_{k} &=&(1-p)^{2}\prod_{i=0}^{k-1}(p^{2^{i+1}}+(p(1-p))^{2^{i}}) \\
&=&(1-p)^{2}p^{2^{k+1}-2}\prod_{i=0}^{k-1}\left( 1+\left( \frac{1-p}{p}%
\right) ^{2^{j}}\right) .
\end{eqnarray*}%
Since $1-p<p$, $\prod_{i=0}^{k-1}\left( 1+((1-p)/p)^{2^{i}}\right) $
converges to a constant $0<$ $c<\infty $. Hence there are positive constants 
$A,B$ such that for large enough $k$ 
\begin{eqnarray*}
\left\Vert T_{0}^{2^{k}}\right\Vert &=&(p(1-p))^{2^{k}}+A_{k}+(1-p)^{2^{k+1}}
\\
&\leq &(p(1-p))^{2^{k}}(1+2c+((1-p)/p)^{2^{k}}) \\
&\leq &A(p(1-p))^{2^{k}}
\end{eqnarray*}%
and similarly 
\begin{equation*}
\left\Vert T_{1}^{2^{k}}\right\Vert \geq Bp^{2^{k+1}}.
\end{equation*}

Let $x_{k}$ be the midpoint of $\Delta _{0}^{(2^{k})}$ and $R_{k}=2\rho
^{5+2^{k+1}}$. Then $R_{k}\geq \ell (\Delta _{0}^{(2^{k})})+\ell (\Delta
_{1}^{(2^{k})}),$ so $B(x_{k},R_{k})\supseteq \Delta _{0}^{(2^{k})}\bigcup
\Delta _{1}^{(2^{k})}$ and therefore 
\begin{equation*}
\mu (B(x_{k},R_{k}))\geq \mu (\Delta _{1}^{(2^{k})})\sim \left\Vert
T_{1}^{2^{k}}\right\Vert \geq Bp^{2^{k+1}}\text{.}
\end{equation*}%
Put $r_{k}=R_{k}^{1+\delta }$ for fixed $\delta >0$. If $k$ is sufficiently
large, then 
\begin{equation*}
r_{k}\leq \rho ^{5+2^{k}+3}/2\leq \ell (\Delta _{0}^{(2^{k})})/2
\end{equation*}%
and therefore $B(x_{k},r_{k})\subseteq \Delta _{0}^{(2^{k})}$. It follows
that%
\begin{equation*}
\mu (B(x_{k},r_{k}))\leq \mu (\Delta _{0}^{(2^{k})})\sim \left\Vert
T_{0}^{2^{k}}\right\Vert \leq A(p(1-p))^{2^{k}}\text{.}
\end{equation*}%
Consequently, 
\begin{equation*}
\frac{\mu (B(x_{k},R_{k}))}{\mu (B(x_{k},r_{k}))}\geq \frac{Bp^{2^{k+1}}}{%
A(p(1-p))^{2^{k}}}=\frac{B}{A}\left( \frac{p}{1-p}\right) ^{2^{k}},
\end{equation*}%
while $R_{k}/r_{k}$ $=R_{k}^{-\delta }=2^{-\delta }\rho ^{-\delta
(5+2^{k+1})}$. Thus 
\begin{equation*}
\overline{\dim }_{\Phi _{\delta }}\mu \geq \frac{\log (p/(1-p))}{2\delta
\left\vert \log \rho \right\vert }
\end{equation*}%
and therefore%
\begin{equation*}
\dim _{qA}\mu =\lim_{\delta \rightarrow 0}\overline{\dim }_{\Phi _{\delta
}}\mu =\infty \text{.}
\end{equation*}
\end{proof}

An equicontractive self-similar measure of finite type is called regular if
the probabilities associated with the left and right-most contractions are
equal and minimal. One example is an $m$-fold convolution of a uniform
Cantor measure on a Cantor set with contraction factor $1/d$ for $d\in 
\mathbb{N}$. Another is a uniform (but not biased) Bernoulli convolution
with contraction factor the inverse of a Pisot number.

\begin{corollary}
Suppose $\mu $ is an equicontractive, self-similar, regular, finite type
measure. Then $\overline{\dim }_{\Phi }\mu <\infty $ whenever $\Phi
(x)\succeq \log |\log x|/|\log x|$ for all $x\leq 1$. In particular, $\dim
_{qA}\mu <\infty $ for such measures $\mu $.
\end{corollary}

\begin{proof}
For such measures $\mu ,$ it is known that $\mu (\Delta _{n})\geq Cn\mu
(\Delta _{n}^{\ast })$, \cite{HHM}, thus $\mu $ is $\Phi $-doubling for such 
$\Phi$.
\end{proof}

The measures studied in Example \ref{FTNotqA} and Proposition \ref{BCbiased}
illustrate the necessity of the hypothesis of regularity. The following
example shows the sharpness of the function $\log |\log x|/|\log x|$.

\begin{example}[An equicontractive, self-similar measure of finite type that
has full support, is regular and has $\overline{\dim }_{\Phi }\protect\mu %
=\infty $ for all $\Phi (x)\ll \log \left\vert \log x\right\vert /\left\vert
\log x\right\vert $]
Consider the IFS, $S_{j}(x)=x/3+d_{j}$ with $d_{0}=0,$ $d_{1}=1/6,$ $%
d_{2}=1/3,$ $d_{3}=2/3$ and probabilities $p_{j}=1/4$ for all $j$. Let $\mu $
be the associated self-similar measure. This example was studied in \cite[%
Ex. 5.11]{HHT}. The two net intervals of level $n$ with endpoint $1/2$ have
length $3^{-n}/2$. The $\mu $-measure of the right interval is at most $%
c_{1}4^{-n}$, while the measure of the left is at least $c_{2}n4^{-n}$ for
some $c_{1},c_{2}>0$. Take $x_{n}$ the midpoint of the right interval, $%
R_{n}=\frac{3}{4}3^{-n}$ and $r_{n}=R_{n}^{1+\Phi (R_{n})}\leq 3^{-n}/4$.
Hence there exist constants $C,\alpha <\infty $ such that 
\begin{equation*}
\frac{c_{2}}{c_{1}}n\leq \frac{\mu (B(x,R_{n}))}{\mu (B(x,r_{n}))}\leq
C\left( \frac{R_{n}}{r_{n}}\right) ^{\alpha }
\end{equation*}%
only if 
\begin{equation*}
\Phi \left( R_{n}\right) \succeq \frac{\log n}{\left\vert \log
R_{n}\right\vert }\succeq \frac{\log \left\vert \log R_{n}\right\vert )}{%
\left\vert \log R_{n}\right\vert }.
\end{equation*}
\end{example}

\section{Discrete Measures}

\label{sec:discrete}

Let $\{a_{n}\}_{n=1}^{\infty }$ be a decreasing sequence tending to $0$ and $%
\{p_{n}\}_{n=0}^{\infty }$ a set of probabilities, $p_n \geq 0$, such that $%
0<\sum_{n=0}^{\infty }p_{n}<\infty $. We define a discrete measure $\mu $
with support $E:=\{a_{n}\}_{n=1}^{\infty }\cup \{0\},$ by%
\begin{equation*}
\mu =\sum_{k}p_{k}\delta _{a_{k}}+p_{0}\delta _{0}\text{.}
\end{equation*}%
Thus $\mu (F)=\sum_{n:a_{n}\in F}p_{n}$ for any Borel set $F\subseteq 
\mathbb{R}\backslash \{0\}$ and $\mu \{0\}=p_{0}$. If we normalize $\mu ,$
then it is a probability measure and normalizing does not change $\Phi $%
-dimensions.

It was shown in \cite{GH} and \cite{GHM1} that if the sequence of gaps $%
\{a_{n}-a_{n+1}\}_{n=1}^{\infty }$ is also decreasing (such as when $%
a_{n}=\beta ^{-n}$ or $n^{-\lambda }$ for $\beta >1$ or $\lambda >0$), then
both the upper Assouad and quasi-Assouad dimensions of $E$ are either $0$ or 
$1$, although not necessarily the same value for the same set $E$. In \cite[%
Example 2.18]{GHM}, it was shown that this need not be true for upper $\Phi $%
-dimensions, even for dimension functions $\Phi $ with upper $\Phi $%
-dimensions lying between the upper quasi-Assouad and Assouad dimensions.
Thus it is natural to ask about the $\Phi $-dimensions for measures
supported on such sets.

As these measures have atoms, their lower $\Phi $-dimensions are always
zero, so it is only the upper $\Phi $-dimensions that are unknown. In \cite%
{FH}, Fraser and Howroyd determined $\dim _{A}\mu $ for such measures $\mu $
when $p_{0}=0$ and either all $p_{n}$ are equal to $n^{-\lambda }$ or all
are equal to $\beta ^{-n}$ for $n\in \mathbb{N}$, and likewise for $a_{n}$
(although with possibly different values for $\lambda $ or $\beta )$. Here,
we will continue to focus on these choices for $p_{n}$ and $a_{n}$,  for $n
\in \mathbb{N}$.

To state our results, it is convenient to let%
\begin{equation*}
L=L_{\Phi }=\limsup_{x\rightarrow 0}\Phi (x)^{-1}\text{ and }\Psi (x)=\frac{%
\log \left\vert \log x\right\vert }{\left\vert \log x\right\vert }.
\end{equation*}%
For $\beta >1$ and $\lambda >0,$ put 
\begin{equation}
s=\frac{\beta -1}{\lambda }\text{ and }t=\frac{\beta }{\lambda +1}\text{ .}
\label{s,t}
\end{equation}%
Note that $s\leq t$ if and only if $t\leq 1$.

\begin{theorem}
\label{P:polypoly}Assume $\mu =p_{0}\delta _{0}+\sum p_{n}\delta _{a_{n}}$
and suppose $\Phi $ is any dimension function.

\begin{enumerate}
\item ``Polynomial-polynomial'': Suppose that for all $n\in \mathbb{N}$, $%
p_{n}=n^{-\beta }$ and $a_{n}=n^{-\lambda }$ for $\beta >1$ and $\lambda >0$%
. If $p_{0}=0,$ then 
\begin{equation*}
\overline{\dim }_{\Phi }\mu =\left\{ 
\begin{array}{cc}
\max (1,s) & \text{if }L\geq \lambda \\ 
\max (t+L(t-s),s) & \text{if }L\leq \lambda%
\end{array}%
\right. ,
\end{equation*}%
while if $p_{0}\neq 0$, then%
\begin{equation*}
\overline{\dim }_{\Phi }\mu =\left\{ 
\begin{array}{cc}
sL+\max (1,s) & \text{if }L\geq \lambda \\ 
(1+L)\max (s,t) & \text{if }L\leq \lambda%
\end{array}%
\right. .
\end{equation*}

\item ``Exponential-exponential'': Suppose that for all $n\in \mathbb{N}$, $%
p_{n}=\beta ^{-n}$and $a_{n}=\lambda ^{-n}$ for $\beta ,\lambda >1$. Then 
\begin{equation*}
\overline{\dim }_{\Phi }\mu =\left\{ 
\begin{array}{cc}
(1+L)\frac{\log \beta }{\log \lambda } & \text{if }p_{0}\neq 0 \\ 
\frac{\log \beta }{\log \lambda } & \text{if }p_{0}=0%
\end{array}%
\right. .
\end{equation*}

\item ``Mixed rates'': (Exponential-polynomial) Suppose that for all $n\in 
\mathbb{N}$, $p_{n}=\beta ^{-n}$and $a_{n}=n^{-\lambda }$ for $\beta >1$ and 
$\lambda >0$. Then 
\begin{equation*}
\overline{\dim }_{\Phi }\mu =\infty. 
\end{equation*}

\item \textquotedblleft Mixed rates\textquotedblright :
(Polynomial-exponential) Suppose that for all $n\in \mathbb{N}$, $%
p_{n}=n^{-\beta }$ and $a_{n}=\lambda ^{-n}$ for $\beta ,\lambda >1$. Then 
\begin{equation*}
\overline{\dim }_{\Phi }\mu =\left\{ 
\begin{array}{cc}
\overline{\lim }_{x\rightarrow 0}\beta \frac{\Psi (x)}{\Phi (x)} & \text{if }%
p_{0}\neq 0 \\ 
\overline{\lim }_{x\rightarrow 0}\frac{\Psi (x)}{\Phi (x)} & \text{if }%
p_{0}=0%
\end{array}%
\right. .
\end{equation*}
\end{enumerate}
\end{theorem}

Before beginning the proof, we will list some immediate corollaries.

\begin{corollary}
{\ }

\begin{enumerate}
\item If $p_{0}\neq 0,$ then $\dim _{A}\mu =\infty $ (in all cases). If $%
p_{0}=0,$ then $\dim _{A}\mu =\infty $ in the mixed rates cases, $\dim
_{A}\mu =\max (1,s)$ in the polynomial-polynomial case and $\dim _{A}\mu
=\log \beta /\log \lambda $ in the exponential-exponential case.

\item The upper quasi-Assouad dimension coincides with the upper Assouad
dimension except in the polynomial-exponential case when $\dim _{qA}\mu =0$
(regardless of the choice of $p_{0}$).

\item If $E=0\cup \{\lambda ^{-n}\}_{n=1}^{\infty }$ and $\Phi (x)/\Psi
(x)\rightarrow \infty $ as $x\rightarrow 0$, then $\overline{\dim }_{\Phi
}E=0$.
\end{enumerate}
\end{corollary}

\begin{proof}
To compute the upper Assouad dimension, just note that when $\Phi =0,$ then $%
L_{\Phi }=\infty $ (so $L\geq \lambda $) and $\overline{\lim }_{x\rightarrow
0}\frac{\Psi (x)}{\Phi (x)}=\infty $. To compute the upper quasi-Assouad
dimension, let $\Phi _{\delta }$ be the constant function $\delta >0,$
observe that $L_{\Phi _{\delta }}\rightarrow \infty $ as $\delta \rightarrow
0$ and use the fact that $\dim _{qA}\mu =\lim_{\delta \rightarrow 0}%
\overline{\dim }_{\Phi _{\delta }}\mu $.

Finally, if $\Phi (x)/\Psi (x)\rightarrow \infty ,$ then, taking $\mu$ as in
the mixed rate case, $0=\overline{\dim }_{\Phi }\mu \geq \overline{\dim }%
_{\Phi }\mathrm{supp}\mu =\overline{\dim }_{\Phi }E$.
\end{proof}

We will give the details of the proof of the theorem in the
polynomial-polynomial case. The other cases require essentially no new ideas
and are less complicated because of the good properties of geometric series
and the fact that exponentials overwhelm polynomials in the asymptotic sense.

We begin with two elementary lemmas.

\begin{lemma}
Under the assumptions and notation of Theorem \ref{P:polypoly}, in the
polynomial-polynomial case%
\begin{equation*}
\mu (B(a_{k},R))\sim \left\{ 
\begin{array}{cc}
\max (R^{s},p_{0}) & \text{if }R>a_{k} \\ 
(a_{k}+R)^{s}-(a_{k}-R)^{s} & \text{if }a_{k}-a_{k+1}<R\leq a_{k} \\ 
a_{k}^{\beta /\lambda } & \text{if }R\leq a_{k}-a_{k+1}%
\end{array}%
\right.
\end{equation*}%
and $\mu (B(0,R))\sim \max (R^{s},p_{0})$.
\end{lemma}

\begin{proof}
These observations follow from the fact that 
\begin{equation*}
\mu (B(a_{k},R))\sim \left\{ 
\begin{array}{cc}
p_{0} & \text{if }R>a_{k}\text{ and }p_{0}\neq 0 \\ 
\sum_{n:a_{n}\in (a_{k}-R,a_{k}+R)}n^{-\beta } & \text{otherwise}%
\end{array}%
\right. .
\end{equation*}%
When $R\leq a_{k}-a_{k+1},$ then $\mu (B(a_{k},R))=\mu \{a_{k}\}=k^{-\beta
}=a_{k}^{\beta /\lambda }$, as claimed.

When $a_{k}-a_{k+1}<R\leq a_{k},$ then choose integers $N\geq k+1$ and $M$
such that $a_{N+1}<a_{k}-R\leq a_{N}$ and $a_{M}\leq a_{k}+R$ $<a_{M-1}$.
(Put $N=\infty $ if $R=a_{k}$.) We have%
\begin{eqnarray*}
\mu (B(a_{k},R)) &=&\sum_{j=M}^{N}j^{-\beta }\sim M^{-\beta +1}-N^{-\beta +1}
\\
&\sim &(a_{k}+R)^{s}-(a_{k}-R)^{s}.
\end{eqnarray*}

The reasoning is similar if $R>a_{k}$.
\end{proof}

\begin{lemma}
Let $s>0$. There are constants $c_{1},c_{2}>0$, depending only on $s$, such
that 
\begin{equation*}
c_{1}a^{s-1}x\leq (a+x)^{s}-(a-x)^{s}\leq c_{2}a^{s-1}x
\end{equation*}%
whenever $0\leq x\leq a$.
\end{lemma}

\begin{proof}
This is clear if $a/2$ $\leq x\leq a$ and otherwise follows quickly from the
Mean value theorem.
\end{proof}

\begin{proof}[Proof of Theorem \protect\ref{P:polypoly}]
We remind the reader that we are proving the polynomial-polynomial case.
Throughout the proof we will use the notation 
\begin{equation*}
X(k,r,R)=\frac{\mu (B(a_{k},R))}{\mu (B(a_{k},r))}\text{ and }X(0,r,R)=\frac{%
\mu (B(0,R))}{\mu (B(0,r))}.
\end{equation*}

\textbf{Step 1: }We will first assume that $x^{\Phi (x)}\rightarrow 0$ as $%
x\rightarrow 0$. \newline

\textbf{Upper bound on }$\overline{\dim }_{\Phi }\mu $: As the arguments are
often quite similar, we will handle the cases $p_{0}=0$ and $p_{0}>0$
concurrently.

Since $X(0,r,R)\sim 1$ if $p_{0}\neq 0$ and comparable to $(R/r)^{s}$ if $%
p_{0}=0$, we easily see that $X(0,r,R)$ $\lesssim (R/r)^{\alpha }$ for any $%
\alpha >0$ when $p_{0}\neq 0,$ and any $\alpha >s$ if $p_{0}=0$. Hence we
now focus on balls centred at $a_{k}$, $k\in \mathbb{N}$.

As it often arises, we will set%
\begin{equation*}
b_{k}=a_{k}-a_{k+1}\sim a_{k}^{(\lambda +1)/\lambda }.
\end{equation*}%
If $R\leq b_{k},$ then also $r\leq b_{k}$ and $X(k,r,R)\sim 1,$ so any $%
\alpha >0$ suffices.

Thus, there remain two cases to study: $R>a_{k}$ and $b_{k}<R\leq a_{k}$.

\textbf{Case 1: }$R>a_{k}$\textbf{.}

\begin{enumerate}
\item Suppose $r>a_{k}$. If $p_{0}\neq 0,$ then $X(k,r,R)\sim 1\lesssim
(R/r)^{\alpha }$ for any $\alpha >0$. If $p_{0}=0,$ then $X(k,r,R)\sim
(R/r)^{s}$.

\item Next, suppose $r\in (b_{k},a_{k}]$. Then, in addition to the
inequality $r\leq R^{1+\Phi (R)},$ we also have%
\begin{equation*}
a_{k}^{(\lambda +1)/\lambda }\lesssim r\leq a_{k}<R.
\end{equation*}%
From the lemmas we know%
\begin{equation*}
X(k,r,R)\sim \frac{\max (p_{0},R^{s})}{(a_{k}+r)^{s}-(a_{k}-r)^{s}\text{ }}%
\sim \frac{\max (p_{0},R^{s})}{a_{k}^{s-1}r\text{ }}.
\end{equation*}

If $s\geq 1,$ then $X(k,r,R)\preceq \max (p_{0},R^{s})r^{-s}$. When $%
p_{0}=0, $ $\alpha >s$ is clearly sufficient to have $X(k,r,R)\lesssim
(R/r)^{\alpha } $. If $p_{0}\neq 0$, it will be enough for $\alpha $ to
satisfy $p_{0}r^{-s}\lesssim (R/r)^{\alpha }$ for small $R$ and all $%
r=R^{1+\Psi (R)}$ with $\Psi (R)\geq \Phi (R)$. Equivalently, we want to
satisfy%
\begin{equation*}
1\lesssim R^{s(1+\Psi (R))-\alpha \Psi (R)},
\end{equation*}%
for small $R$, hence $\alpha >s(L+1)$ is enough.

When $s<1$ (equivalently, $s<t),$ then $a_{k}^{s-1}r\succeq r\left( \min
(R,r^{\lambda /(\lambda +1)})\right) ^{s-1}$. If this minimum is $R,$ (which
can occur only if $\Phi (R)\leq 1/\lambda ),$ then 
\begin{equation*}
X(k,r,R)\lesssim \frac{\max (p_{0},R^{s})}{R^{s-1}r\text{ }}=\left\{ 
\begin{array}{cc}
p_{0}R^{1-s}r^{-1} & \text{if }p_{0}\neq 0 \\ 
Rr^{-1} & \text{if }p_{0}=0%
\end{array}%
\right. .
\end{equation*}%
Again, putting $r=R^{1+\Psi (R)}$ with $\Psi \geq \Phi ,$ it is easy to
check that the requirement $X(k,r,R)\lesssim (R/r)^{\alpha }$ is satisfied
with $\alpha >1$ when $p_{0}=0$ and with $\alpha >1+sL$ when $p_{0}\neq 0$.

If, instead, $\min (R,r^{\lambda /(\lambda +1)})=r^{\lambda /(\lambda +1)}$,
then we have $r=R^{1+\Psi (R)}$ with $\Psi (R)\geq \max (\Phi (R),1/\lambda )
$. Moreover $a_{k}^{s-1}r\gtrsim r^{t}$, thus $X(k,r,R)\lesssim \max
(p_{0},R^{s})r^{-t}$. If $p_{0}=0,$ it will be enough to have $%
R^{s}r^{-t}\lesssim (R/r)^{\alpha }$, and this happens if  
\begin{equation*}
\alpha >t+(t-s)/\Psi (R).
\end{equation*}%
If $L>\lambda ,$ then $\Phi (R)<1/\lambda $ for small enough $R,$ so $1/\Psi
(R)\leq \lambda $. Thus $\alpha >t+(t-s)\lambda =1$ suffices. Similarly, $%
\alpha >t+L(t-s)$ is sufficient when $L\leq \lambda $. If $p_{0}\neq 0$, we
will want $p_{0}r^{-t}\lesssim (R/r)^{\alpha }$ and this is satisfied by any 
$\alpha >\beta $ if $L>\lambda ,$ and for any $\alpha >t(1+L)$ when $L\leq
\lambda $.

Here is a summary of the choices of $\alpha $ for which $X(k,r,R)\leq
c(R/r)^{\alpha }$ in case 1(ii): If $p_{0}=0,$ then it is sufficient to have%
\begin{equation*}
\alpha >\left\{ 
\begin{array}{cc}
s & \text{if }s\geq 1 \\ 
1 & \text{if }s<1\text{ and }L\geq \lambda \\ 
t+L(t-s) & \text{if }s<1\text{ and }L<\lambda%
\end{array}%
\right. .
\end{equation*}%
If $p_{0}\neq 0$, then we can take 
\begin{equation*}
\alpha >\left\{ 
\begin{array}{cc}
s(L+1) & \text{if }s\geq 1 \\ 
\max (1+sL,\beta ) & \text{if }s<1\text{ and }L\geq \lambda \\ 
t(1+L) & \text{if }s<1\text{ and }L<\lambda%
\end{array}%
\right. .
\end{equation*}

\item Otherwise, $r\leq b_{k}\sim a_{k}^{1+1/\lambda }$, say $r=R^{1+\Psi
(R)}$ where $\Psi (R)\geq \max (\Phi (R),1/\lambda \dot{)}$. In this case 
\begin{equation*}
X(k,r,R)\lesssim \max (p_{0},R^{s})a_{k}^{-\beta /\lambda }\lesssim \max
(p_{0},R^{s})r^{-t}\text{. }
\end{equation*}%
If $p_{0}=0$, it suffices to have $\alpha >t+(t-s)/\Psi (R)$. If $s\geq 1,$
(equivalently, $s\geq t$) this inequality is satisfied with any $\alpha >t$,
while if $s<1$ the reasoning is similar to the arguments in case (ii).
Likewise, the reasoning when $p_{0}\neq 0$ is similar to case (ii).

To summarize: If $p_{0}=0,$ it is enough to have%
\begin{equation*}
\alpha >\left\{ 
\begin{array}{cc}
t & \text{if }s\geq 1 \\ 
1 & \text{if }s<1\text{ and }L\geq \lambda \\ 
t+L(t-s) & \text{if }s<1\text{ and }L<\lambda%
\end{array}%
\right. ,
\end{equation*}%
while if $p_{0}\neq 0,$ then we can take%
\begin{equation*}
\alpha >\left\{ 
\begin{array}{cc}
\beta & \text{if }L\geq \lambda \\ 
t(1+L) & \text{if }L<\lambda%
\end{array}%
\right. .
\end{equation*}
\end{enumerate}

\textbf{Case 2: }$b_{k}<R\leq a_{k}$\textbf{. }Here the calculations are the
same regardless of the choice of $p_{0}$.

\begin{enumerate}
\item Suppose $r\leq b_{k}$, say $r=R^{1+\Psi (R)}$ where $\Psi (R)\geq \Phi
(R)$. Here, $\mu (B(a_{k},r))\sim a_{k}^{\beta /\lambda }$, so as $%
a_{k}^{-1}\lesssim \min (R^{-1},r^{-\lambda /(\lambda +1)}),$ 
\begin{equation*}
X(k,r,R)\sim a_{k}^{s-1-\beta /\lambda }R\lesssim R\min
(r^{-1},R^{-(1+1/\lambda )}).
\end{equation*}%
By consideration of the two possible choices for the minimum, it can be
checked that $\alpha >\min (1,L/\lambda )$ will suffice.

\item Otherwise, $b_{k}<r<R\leq a_{k}$ (a choice which can only occur if $%
L\geq \lambda ),$ and then it is easy to see that $\alpha >1$ is sufficient,
so again $\alpha >\min (1,L/\lambda )$ will work.
\end{enumerate}

It is a tedious exercise to check that these constraints on $\alpha $ imply
that the values specifed in the Proposition are upper bounds on $\overline{%
\dim }_{\Phi }\mu $. \newline

\textbf{Lower bound on }$\overline{\dim }_{\Phi }\mu $: We turn now to
proving $\overline{\dim }_{\Phi }\mu $ is as large as claimed.

First, suppose $p_{0}=0$. In this case, 
\begin{equation*}
\dim _{\mathrm{loc}}\mu (0)=\lim_{n}\frac{\log \mu (B(0,n^{-\lambda }))}{%
\log n^{-\lambda }}\sim \frac{(\beta -1)}{\lambda }=s,
\end{equation*}%
so it is certainly true that $\overline{\dim }_{\Phi }\mu \geq s$ for all
choices of $\Phi $.

Essentially the same arguments as in \cite{FY}, show that if $\Phi $ is the
constant function $1/\theta -1$, then 
\begin{equation*}
\overline{\dim }_{\Phi }E=\min \left( 1,\frac{1}{(1+\lambda )(1-\theta )}%
\right) .
\end{equation*}%
As $\overline{\dim }_{\Phi }\mu \geq \overline{\dim }_{\Phi }E,$ it follows
that if $L\geq \lambda ,$ then, also, $\overline{\dim }_{\Phi }\mu \geq 1$.
Hence if $L\geq \lambda ,$ then $\overline{\dim }_{\Phi }\mu \geq \max (1,s)$%
.

Next, suppose $0\leq L<\lambda $ and $t>s$ (for otherwise, $\max
(t+L(t-s),s)=s$). Choose $R_{j}\rightarrow 0$ such that $\Phi
(R_{j})\rightarrow 1/L$ and put $r_{j}=R_{j}^{1+\Phi (R_{j})}$. Choose $%
k=k_{j}$ such that $a_{k}-a_{k+1}\geq r>a_{k+1}-a_{k+2}$. Since $\Phi
(R_{j})>1/\lambda $ for large $j,$ one can check that $R_{j}>a_{k}$ and
hence 
\begin{equation*}
X(k,r_{j},R_{j})\sim \frac{R_{j}^{s}}{a_{k}^{\beta /\lambda }}\succeq
R_{j}^{s-t(1+\Phi (R_{j}))},
\end{equation*}%
while $(R_{j}/r_{j})^{\alpha }\sim R_{j}^{-\alpha \Phi (R_{j})}$. Thus in
order to satisfy $X(k,r_{j},R_{j})\preceq (R_{j}/r_{j})^{\alpha }$ for all $%
j,$ we require 
\begin{equation*}
1\preceq R_{j}^{-\Phi (R_{j})(\alpha -(t+(t-s)/\Phi (R_{j})))}.
\end{equation*}%
Since $R_{j}^{-\alpha \Phi (R_{j})}\rightarrow \infty $ and $\Phi
(R_{j})\rightarrow 1/L,$ we see that $\alpha \geq t+L(t-s)$ is a necessary
condition.

Now assume $p_{0}\neq 0$ and first suppose $0\leq L<\lambda $. With the same
choice of $r_{j},R_{j}$ and $a_{k}$ as above, we have $X(k,r_{j},R_{j})\sim
a_{k}^{-\beta /\lambda }$. It follows that we require $\alpha \geq t(1+L).$

If, instead, we pick $k=k_{j}$ such that $a_{k}\leq R_{j}^{1+\Phi
(R_{j})}<a_{k-1}$ and put $r_{j}=a_{k},$ then $r_{j}\leq R_{j}^{1+\Phi
(R_{j})}$. With these choices for $k,r_{j},R_{j}$, we have 
\begin{equation*}
X(k,r_{j},R_{j})\sim p_{0}a_{k}^{-s}\text{ }\sim R_{j}^{-s(1+\Phi (R_{j}))}%
\text{ },
\end{equation*}%
and one can deduce that $\alpha \geq s(L+1)$ is a necessary condition to
satisfy $X(k,r_{j},R_{j})\preceq (R_{j}/r_{j})^{\alpha }$.

Lastly, assume $L\geq \lambda $. Put $r_{j}=R_{j}^{1+\Phi (R_{j})}$ and
choose $k=k_{j}$ such that $a_{k}<R_{j}\leq a_{k-1}$. For large $j,$ $%
r_{j}\geq a_{k}-a_{k+1}\sim a_{k}^{1+1/\lambda }$. Since $%
r_{j}/R_{j}=R_{j}^{\Phi (R_{j})}\rightarrow 0$, we can assume $r_{j}<a_{k}$.
Hence $X(k,r_{j},R_{j})\sim p_{0}R_{j}^{-(s+\Phi (R_{j})}$ and it follows
that $\alpha \geq 1+sL$ is required.

This completes the proof that $\overline{\dim }_{\Phi }\mu $ is as claimed
in the statement of the Theorem for the polynomial-polynomial case when $%
x^{\Phi (x)}\rightarrow 0$. \newline

\textbf{Step 2:} We now consider the case that $\delta
=\limsup_{x\rightarrow 0}x^{\Phi (x)}>0$ or, equivalently, $%
\limsup_{x\rightarrow 0}\Phi (x)\left\vert \log x\right\vert <\infty $.

For $p_{0}>0,$ choose $R_{j}\rightarrow 0$ such that $R_{j}^{\Phi
(R_{j})}\rightarrow \delta >0$. Pick $k=k_{j}$ such that $a_{k}<R_{j}\leq
a_{k_{j-1}}$ and let $r_{j}=\min (R_{j}^{1+\Phi (R_{j})},a_{k})$, so that $%
R_{j}/r_{j}\sim 1$. If $N_{j}$ is chosen such that $a_{N_{j}+1}\leq
a_{k}+r_{j}\leq a_{N_{j}},$ then $\mu (B(a_{k},r_{j}))\leq
\sum_{N_{j}}^{\infty }p_{i}\rightarrow 0$ as $j\rightarrow \infty $. But $%
\mu (B(a_{k},R_{j}))\geq p_{0}$. Consequently, $X(k,r_{j},R_{j})\rightarrow
\infty $ as $j\rightarrow \infty $ and forces $\overline{\dim }_{\Phi }\mu
=\infty $ for all such $\Phi $.

So, suppose $p_{0}=0$ and define 
\begin{equation*}
\Psi _{0}(x)=\frac{\sqrt{\log \left\vert \log x\right\vert }}{\left\vert
\log x\right\vert }\text{ and }\Phi _{0}(x)=\max \left( \Psi _{0}(x),\Phi
(x)\right) .
\end{equation*}%
Then $\Phi _{0}\in \mathcal{D}$ and as $\Phi _{0}(x)\geq \Psi _{0}(x)$ for
all $x,$ $x^{\Phi _{0}(x)}\rightarrow 0$. Furthermore, $\Phi (x)\leq \Phi
_{0}(x)$, hence $\overline{\dim }_{\Phi _{0}}\mu \leq \overline{\dim }_{\Phi
}\mu \leq \dim _{A}\mu $. Since $\Phi _{0}\leq \Phi +\Psi _{0}$, one can
verify that $L_{\Phi _{0}}=\infty $ and thus by the first part of the
theorem, $\overline{\dim }_{\Phi _{0}}\mu =\max (1,s)$. It was shown in \cite%
{FH} that $\dim _{A}\mu =\max (1,s)$ and hence also $\overline{\dim }_{\Phi
}\mu =\max (1,s),$ as claimed in the statement of the Theorem for the
polynomial-polynomial case.
\end{proof}

\begin{remark}
The choice of $\Psi _{0}$ was made to ensure that the arguments for the
other choices of $p_{n}$ and $a_{n}$ are virtually the same in the final
steps of the proof.
\end{remark}

\end{document}